\let\ORIlabel\label
\let\ORIrefstepcounter\refstepcounter
\AddToHook{package/hyperref/before}{
   \let\label\ORIlabel 
   \let\refstepcounter\ORIrefstepcounter}

\documentclass{siamart220329}
\usepackage{amsmath,amsfonts,amssymb}
\usepackage{graphicx}

\newtheorem{example}{Example}
\newtheorem{proposition1}{Proposition}[section]

\begin{document}

\title{Accuracy Analysis of the Proxy Point Method with Applications to Some Toeplitz Matrices\thanks{Submitted for review.\funding{The research of Mikhail Lepilov was supported in part by NSF grant DMS-2038118.}}}
\author{Mikhail Lepilov\thanks{Department of Mathematics, Rensselaer Polytechnic Institute, Troy, NY 12180 (lepilm@rpi.edu).}\and Jianlin Xia\thanks{Department of Mathematics, Purdue University, West Lafayette, IN 47907 (\mbox{jianlinxia@yahoo.com}).}}
\maketitle

\begin{abstract}
%Fast low-rank matrix approximation methods are invaluable in many areas of scientific computations. 
For some kernel matrices, low-rank approximations can be quickly obtained via analytic techniques. One important class of analytic methods that has received attention in recent years is based on the use of proxy points. Accuracy analysis for various proxy point methods has often been heuristic in nature, other than for certain special kernels. For more general cases, the methods lack an explicit number or location of proxy points required to yield a particular approximation accuracy. In this work, we carry out new analysis of a proxy point method that is applicable to general complex-analytic kernels. An intuitive way of choosing proxy points is used to show explicit error bounds. Such bounds decay exponentially with regard to the number of proxy points. This also leads to convenient estimates of numerical ranks of relevant kernel matrices. To showcase the utility of this new analysis, we apply it to design a new sublinear-time hierarchically semiseparable approximation method for certain Toeplitz matrices, including ones that frequently arise from real-world applications. This allows, for example, inversion of such matrices with lower computational complexity compared with existing direct methods. Some extensions of these ideas are also discussed.
%Furthermore, this new proxy point accuracy analysis also has applications to establishing matrix rank bounds. Finally, it may be readily generalized to several variables.
\end{abstract}

\headers{Mikhail Lepilov and Jianlin Xia}{Proxy Point Method Analysis and Applications}
\begin{keywords}
kernel matrix, proxy point method, approximation error, rank-structured matrix, Toeplitz matrix, sublinear complexity
\end{keywords}

\begin{AMS}
65E99, 65F55, 65G99
\end{AMS}

\section{Introduction}
\label{sec:intro}
Low-rank matrix approximation is a common task in many areas of mathematics, computation, and engineering. By a low-rank approximation of an $m\times n$ matrix $A$, we mean a decomposition $A\approx UV^T$, such that the column size of $U$ is much smaller than $m$ and $n$, and such that the approximation satisfies a certain accuracy requirement. Such an approximation allows even very large matrices to be represented, up to a certain accuracy, by matrices on which operations may be carried out significantly faster. Broadly speaking, we may divide (deterministic) low-rank approximation techniques into two classes: algebraic methods and analytic methods. Examples of algebraic methods include truncated SVDs and rank-revealing QR factorizations. Such methods are applicable to any matrix but are typically too slow to apply to large matrices. Examples of analytic methods include Taylor series expansions, interpolations, and proxy point methods. Such methods can be typically applied with little or essentially no computational cost, but require the matrix under consideration to be a kernel matrix defined by a kernel with desirable analytic properties.

In this work, we focus on one proxy point method which uses a set of proxy points to quickly construct basis matrices in low-rank approximations of kernel matrices. The roots of this idea may be traced back to the fast multipole method (FMM) \cite{gre87} and its variants \cite{mar07,yin04}. The error introduced by an analytic low-rank approximation is governed by analytic properties of the kernel in question. 
% For the aforementioned Taylor approximation-based methods, such analysis is carried out, for example, in \cite{mhs}. More precise statements can also be given with classical approximation techniques using the modulus of continuity of the kernel. 
For proxy point methods, most accuracy analysis has relied on heuristics, arguing exponential convergence in the number of proxy points. Such analysis may be found in \cite{mar05-1,yin04,cauchyfmm}. This is contrasted with the type of analysis carried out for the proxy point method in \cite{kercompr}, which relates the number and location of proxy points to accurate error bounds of the low-rank approximation of the kernel matrix. However, the analysis in this last reference is only applicable to kernels of the form $k(x,y)=1/(x-y)^d$ for $x,y\in\mathbb{C}$ and $d\in\mathbb{N}$.

Here, we generalize this latter type of analysis to any one-dimensional complex-analytic kernel, bounding the error of the proxy point method. Since for well-separated sets, many commonly used kernels are complex-analytic in each variable on a region containing the set (holding the other variable constant), this analysis applies widely to one-dimensional kernels. We also show that, if such a kernel satisfies a certain univalence criterion, we may select proxy points in a manner similar to the work in \cite{kercompr} so as to get a rigorous accuracy bound that decays exponentially with respect to the number of proxy points. This also yields a convenient numerical rank estimate for the associated kernel matrix.
%the resulting accuracy bound on its proxy point approximation shows that number of proxy points required to guarantee a desired approximation tolerance is logarithmic in the tolerance.

In addition to its theoretical utility, such analysis is useful, for example, when performing interpolative decompositions \cite{lib07} of matrices using function approximation by a proxy point method. In \cite{xin20}, for example, even though the error introduced in the function approximation step is used to provide a bound for the error incurred in the interpolative decomposition step, the function approximation error is not itself studied. In this way, we hope that this work helps to fill gaps in the existing literature on proxy point methods.

As another illustration of the utility of this analysis, and following our work in \cite{hypercauchy}, we introduce a new sublinear-time hierarchically semiseparable (HSS) approximation algorithm for certain Toeplitz matrices arising from univalent maps applied to regular grids. Such matrices appear, say, as covariance matrices of Gaussian processes like in \cite{gausstoep}. A general overview of fast direct computations with Gaussian process covariance matrices is given in \cite{gaussproc}. In this context, the kernel matrix is obtained from applying a given positive-definite kernel to a regular 1-dimensional grid.
Existing rank-structured (HSS or similar) approximation construction schemes for such matrices carry at least $O(n)$ costs \cite{gaussproc, toep,toeprs}, so our scheme represents a substantial speedup over existing methods for general Toeplitz matrices such as those recently studied in \cite{bkw}.
% In particular, for applicable matrices, this provides an asymptotically faster direct solver than previously shown. This is because the fastest existing solution schemes solve an equivalent system in Fourier space using structured matrices, requiring the application of the Fast Fourier Transform (FFT) to the solution at the end. By contrast, we avoid this $O(n\log n)$ FFT time cost, allowing for an asymptotically faster solution. 
% The aforementioned univalence criterion then allows us to decide which given matrix may be compressed with this scheme. Matrices to which the new scheme applies include certain one-dimensional Gaussian kernel matrices formed on a regular grid, an example of which appears in \cite{gausstoep}, among other settings.

Additionally, we propose an extension of the ideas to kernels of dimension greater than one.
We also verify that our proxy point approximation error bounds are illustrative of real-world performance by carrying out several numerical tests. Tests on the accuracy bounds for the proxy point low-rank approximation applied to different sets of points and kernels are given. We also carry out tests of the above HSS compression scheme applied to specific matrices.

The rest of this paper is structured as follows. In Section~\ref{sec:proxy}, we go over the proxy point method and perform the new proxy point error analysis for general complex-analytic functions. We also show how to use this analysis to guarantee the efficacy of proxy point approximations to some Toeplitz matrices. In Section~\ref{sec:toeplitzconst}, we briefly review HSS matrix approximation in order to detail the HSS construction algorithm for the Toeplitz matrices under our consideration. In Section~\ref{sec:numerical}, we perform some numerical tests. Finally, we suggest some extensions of this work in Section~\ref{sec:ext}.

Throughout the paper, we use the following notation.
\begin{itemize}
\item Let $c\in\mathbb{C}$ and $r>0$. Then $\mathcal{B}(c,r)$ denotes the open ball in $\mathbb{C}$ with center $c$ and radius $r$, $\mathcal{O}(\mathcal{B}(c,r))$ denotes the set of holomorphic functions on $\mathcal{B}(c,r)$, and $\mathbb{D}$ denotes $\mathcal{B}(0,1)$.
\item For integers $i\leq j$, $[i,j]$  denotes the set  $\{i,i+1,\ldots,j\}$.
%$[a,b]=\{j\in\mathbb{N}:\;a\leq j\leq b\}$. (In particular, with this notation we will never mean the closed interval on the real line from $a$ to $b$.) 
\item Let $k:F\times G\to\mathbb{C}$ be a function and $X\subseteq F,Y\subseteq G$ be totally-ordered finite subsets of size $r$ and $s$, respectively. Then $k(X,Y)=(k\left(x_i,y_j\right))_{r\times s}$ means the $r\times s$ matrix with $(i,j)$ entry $k\left(x_i,y_j\right)$, where $x_i$ is the $i$th element of $X$ and $y_j$ is the $j$th element of $Y$.
\item Let $C$ be an $m\times n$ matrix and $M\subseteq\{1:m\},N\subseteq\{1:n\}$. Then by $A_{M\times N}$ we mean the $|M|\times|N|$ submatrix of $A$ picked by the row index set $M$ and column index set $N$.%consisting of entries $B_{j,k}=A_{l_j,m_k}$, where $l_j$ is the $j$th element of $L$ and $m_k$ is the $k$th element of $M$, ordered the usual way.
\end{itemize}

\section{Accuracy analysis for the proxy point method}
\label{sec:proxy}
For a kernel matrix $k(X,Y)$ defined by the evaluation of a kernel function $k(x,y)$ at two finite sets $X,Y\subseteq\mathbb{C}$, the proxy point method is a simple yet powerful way for finding a low-rank approximation. The idea is to pick an appropriate set of proxy points $Z\subseteq\mathbb{C}$ based on ideas from, for example, potential theory, function interpolation, or an integral representation \cite{kre14,mar07,xin20,kercompr,yin04}. A low-rank approximation then carries the form
\begin{equation}
    k(X,Y)\approx UV^T,\quad \text{with}\quad V=k(Y,Z).\label{eq:lr}
\end{equation}
(Alternatively, the form may be $UV^T$ with $U=k(X,Z)$, depending on the context.) In particular, it is shown in \cite{kercompr} that one way of approximating $k(x,y)=1/(x-y)^d$, for $d\in\mathbb{N}$, with an integral representation is to use the Cauchy integral formula. This would then allow us to take $Z$ to be a set of quadrature points. In this work, we expand this idea to more general kernels.

\subsection{Accuracy of kernel function approximation}
We follow the strategy of \cite{kercompr} but derive the approximation error results for general complex-analytic kernels $k(x,y)$. 
%First, we review the proxy point method and provide a new bound for the error introduced in its application. For a detailed discussion of this method applied to a specific kernel function, as well as its error analysis, see \cite{kercompr}. 
Let $D=\mathcal{B}(c,r)$ and $E=\mathcal{B}(c,R)$ be open balls in $\mathbb{C}$, with $r<R$. Let $X=\{x_j\}_{j=1}^m\subseteq D$ and $Y=\{y_j\}_{j=1}^n$ be finite sets, and let $k:\mathbb{C}\times Y\to\mathbb{C}$ be a function such that, for each $y\in Y$,  $k(z,y)$ is an analytic function of $z$ on $E$.
%\textcolor{red}{[???]}
Then, for each $x\in X,y\in Y$, by the Cauchy integral formula, we have
\begin{equation}
k(x,y)=\frac{1}{2\pi i}\int_{C}\frac{k(\zeta,y)}{\zeta-x}d\zeta,\label{eq:int}
\end{equation}
where $C$ is the boundary of an open ball with center $c$ and radius $\sqrt{Rr}$. See Figure~\ref{fig:proxy}. Here, we chose the radius $\sqrt{Rr}$ heuristically from the study in \cite{kercompr} (that was conducted for a special kernel only). We will see shortly that precisely this choice of radius allows the analysis of this section to give a good accuracy bound for our proxy point approximation to $k(x,y)$.

\begin{figure}[ptbh]
\centering
\tabcolsep5mm
\includegraphics[height=1.8in]{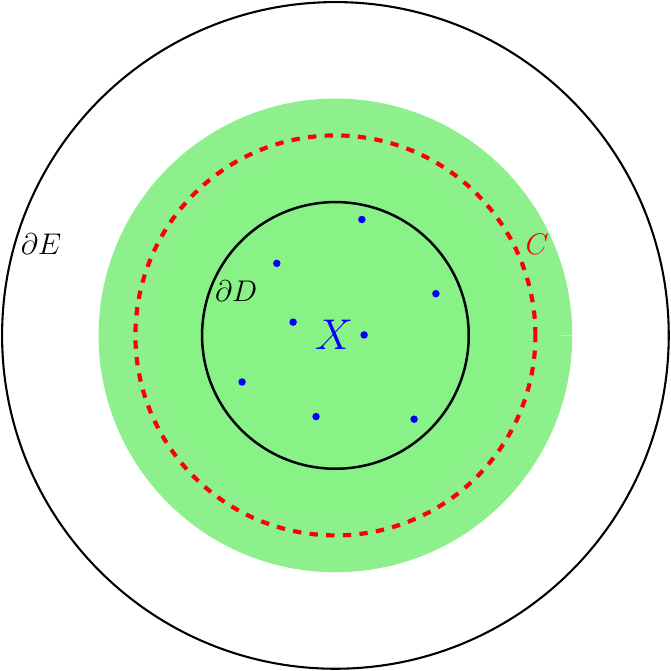}
% \begin{tabular}[c]{p{2.1in}p{2.1in}}
% \centering
% \includegraphics[height=2in]{proxy} & \includegraphics[height=2in]{proxy2}\\
% \vspace{.1in}{\small (a) The contour $C$ (dashed), the finite set $X$, and the boundaries of the open balls $D$ and $E$. The shaded green region shows the annulus $F'$ used in the proof of Proposition~\ref{prop:proxelwisebound}, which is a compact region where $k_{x,y}(z)$ is analytic in $z$ for all $x\in X$ and $y\in Y$.
% \%($k_{x,y}(z)$ is actually assumed to be analytic in $z$ on all of $E\setminus D$, but the proof requires a compact region.)
% }& \vspace{.1in}{\small (b) The finite sets $X$  and $Z~=~\bigcup_{j=1}^Nz_j$ (dots), as well as the boundaries of the open balls $D$ and $E$.}\end{tabular}
\caption{The contour $C$ (dashed), the finite set $X$, and the boundaries of the open balls $D$ and $E$. (The set $Y$, which parametrizes the domain of analyticity in $z$ of $k(z,y)$, is not pictured.) The shaded green region shows the disk $F$ defined in Proposition~\ref{prop:proxelwisebound}, which is a compact region where the maximum of $k(z,y)$ (over all $y\in Y$) determines the approximation bound.
%($k_{x,y}(z)$ is actually assumed to be analytic in $z$ on all of $E\setminus D$, but the proof requires a compact region.)
}\label{fig:proxy}
\end{figure}
%{\textcolor{blue}{[As another example where proper transition/explanation is needed: in the caption of the figure, $k_{x,y}(z)$ and $A$ are undefined. They are defined later in the proof of prosition 2.1, but this figure appears earlier. BTW, it is better to not use $A$ for the annulus since later it will be used for a matrix. Moreover, it says $k(x,y)$ is analytic within the annulus region. But it then also says it is actually assumed to be analytic in $E\setminus D$.  Why such discrepancy? Moreover, in the 1st paragraph of this section and in proposition 2.1, it says it is analytic on $E$. --- {\bf all such inconsistent/unmotivated statements are very confusing}]}}

Using the trapezoidal rule with $p$ points to approximate \eqref{eq:int}, we then have
\begin{equation}
k(x,y)=\frac{\sqrt{Rr}}{p}\sum_{j=1}^{p}\left(\frac{1}{z_j-x}\right)\left(\omega^jk(z_j,y)\right)+\epsilon,
\label{eq:proxyapprox}
\end{equation}
where
\begin{equation}
    z_j=c+\sqrt{Rr}\omega^j,\quad  \omega=e^{\frac{2\pi i}{p}},\label{eq:pp}
\end{equation}
and the termwise error $\epsilon\in\mathbb{C}$ ideally has very small magnitude. The points $z_j$ will serve as our proxy points. That is, $Z~=\{z_j\}_{j=1}^p$ in (\ref{eq:lr}). (\ref{eq:proxyapprox}) provides a separable, or degenerate, expansion for $k(x,y)$. Hence, in this setup, for the low-rank approximation \eqref{eq:proxyapprox}, we take
\begin{equation*}
    U=-\frac{\sqrt{Rr}}{p}\left(\frac{\omega^j}{x_i-z_j}\right)_{m\times p}=\frac{1}{p}\left(\frac{c-z_j}{x_i-z_j}\right)_{m\times p},\quad V=k(Y,Z).
\end{equation*}
(Notice $m=|X|$ and $p=|Z|$.) This yields a proxy point low-rank approximation to $k(X,Y)$.

Note that in the setup for \eqref{eq:int} above and in Figure~\ref{fig:proxy}, the geometric role of the set $Y$ is not explicitly mentioned or shown, even though the exact locations of the points in $Y$ are a key part of the analysis to follow. Indeed, in the existing literature, $Y$ is typically assumed to be well separated from $X$ and the distance between $X$ and $Y$ is used for certain bounds or heuristics \cite{hypercauchy,kercompr}. Here, the geometric information of the set $Y$ factors into the above by determining the domain of analyticity (in $z$) of $k(z,y)$, which in turn determines $R$ and $r$, given $c$. It also determines the disk $F=\mathcal{B}(c,r\left(R/r\right)^{3/4})$, shown in Figure~\ref{fig:proxy} and used to quantify the growth of $k$ in the bound of Proposition~\ref{prop:proxelwisebound}.

For example, consider the kernel $k(x,y)=1/(x-y)$, and fix $c=0$. $R$ is chosen to be the largest radius such that $\mathcal{B}(0,R)$ does not contain a point of $Y$.  So, if the closest point to $c$ in $Y$ is $y_j=i(=\sqrt{-1})$,
%{\textcolor{blue}{[Should $i$ be 1? Why not simply say that $Y$ is outside $\partial E$, at least for this Cauchy kernel?]}}
then we may pick $E=\mathcal{B}(0,1)$ and therefore obtain $R=1$. On the other hand, if the closest point in $Y$ to $c$ is $y_j=2$, we may pick $E=\mathcal{B}(0,2)$ and therefore $R=2$. Similarly, $r$ is chosen to be the smallest radius such that $\mathcal{B}(0,r)$ contains all the points in $X$.
We will see that maximizing the ratio $R/r$ in this manner enables the analysis below to give the best asymptotic proxy point approximation bound, given a tame growth of $k(z,y)$ in $z$ on $F$. This matches the heuristic explored in \cite{kercompr}.
%{\textcolor{blue}{[By 'tightest', it sounds like this is the best among any possible bounds.]}}

Our first goal is to find a bound for $\vert\epsilon\vert$ in \eqref{eq:proxyapprox}. To do so, we give the following result. Its proof is based on classical techniques, including those for the proof of \cite[Theorem 2.2]{trefethen}, with some modifications for the kernel matrix context.

%{\textcolor{blue}{[Since $Y$ is not shown but $E$ is, it is better to explicitly state the role of $E$ and why $Y$ is not shown. Maybe also mention possible locations of $Y$. Currently, the reader may be confused. Also, the symbol $k$ is used for subscripts somewhere, better to replace it by $\kappa$]}}

%{\textcolor{blue}{[It is not very formal to use semicolons to separate multiple complete sentences like "let ...; let ..." in propositions, or even elsewhere in a publication. Please fix all such usage in the paper]}}
\begin{proposition1}
\label{prop:proxelwisebound}
Let $D=\mathcal{B}(c,r)$ and $E=\mathcal{B}(c,R)$ be open balls in $\mathbb{C}$, with $r<R$, and let $X\subseteq D$ and $Y$ be finite sets. Given $p$, let each $z_j$ for $1\leq j\leq p$ be defined as in (\ref{eq:pp}), and let $k:\mathbb{C}\times Y\to\mathbb{C}$ be a function such that, for each $y\in Y$, $k(z,y)$ is an analytic function of $z$ on $E$. Then for each $x\in X,y\in Y$, we have
\begin{equation}
\label{eq:termwisebound}
|\varepsilon|=\left\vert k(x,y)-\frac{\sqrt{Rr}}{p}\sum_{j=1}^{p}\left(\frac{1}{z_j-x}\right)\left(\omega^jk(z_j,y)\right)\right\vert\leq \alpha\frac{\max_{z\in\partial F}\left\vert k(z,y)\right\vert}{\left(R/r\right)^{p/4}-1},
\end{equation}
where $\alpha=2\frac{(R/r)^{1/4}}{(R/r)^{1/4}-1}$ and $F=\mathcal{B}(c,r\left(R/r\right)^{3/4})$.
\end{proposition1}
\begin{proof}
Fix $x\in X$ and $y\in Y$. First, by the parametrization $\gamma(t)=c+e^{ti}\sqrt{Rr}$ of the contour $C$ in \eqref{eq:int}, we may write
\[
k(x,y)=\frac{1}{2\pi i}\int_0^{2\pi}\frac{k(c+\sqrt{Rr}e^{ti},y)(i\sqrt{Rr}e^{ti})}{c+\sqrt{Rr}e^{ti}-x}dt.
\]
Define $k_{x,y}:\mathcal{B}(0,\sqrt{R/r})\setminus\overline{\mathcal{B}(0,\sqrt{r/R})}\to\mathbb{C}$ by
\begin{equation}
k_{x,y}(z)=\frac{k(c+z\sqrt{Rr},y)(z\sqrt{Rr})}{c+z\sqrt{Rr}-x}.\label{eq:kxy}
\end{equation}
Then
\begin{equation}
k(x,y)=\frac{1}{2\pi i}\int_0^{2\pi}\frac{k_{x,y}(e^{ti})ie^{ti}}{e^{ti}}dt
=\frac{1}{2\pi i}\int_{\partial \mathcal{B}(0,1)}\frac{k_{x,y}(\zeta)}{\zeta}d\zeta\equiv a_0,\label{eq:kxyquad}
\end{equation}
where $a_0$ denotes the $0$th Laurent coefficient of $k_{x,y}$.

Consider the Laurent expansion of $k_{x,y}(z)$ at $0$:
\begin{equation}
    k_{x,y}(z)=\sum_{l=-\infty}^\infty a_lz^l,
    \label{eq:laurentexp}
\end{equation}
which, by our assumption on $k$, is valid everywhere that $k_{x,y}$ is defined.
%{\textcolor{blue}{[How does this specific annulus come up? It does not look like there is a particular reason for these two specific radii? Also, fig 2.1 says $k_{x,y}$ is analytic on $F'$, while it should be this one. $F'$ is the version when mapping back to the original sets. There is confusion in the annuli. BTW, the symbol $A$ was used on multiple different occasions.]}}
We show that the sum of some Laurent coefficients $|a_l|$ can be used to bound the error incurred in applying the trapezoidal quadrature rule to \eqref{eq:kxyquad} as in \eqref{eq:proxyapprox}. In fact,  \eqref{eq:proxyapprox}, \eqref{eq:pp}, and \eqref{eq:kxy} mean
\begin{equation*}
    \varepsilon=k(x,y)-\frac 1p\sum_{j=1}^{p} k_{x,y}(\omega^j).
\end{equation*}

For some steps below, we need a compact subregion on which the expansion \eqref{eq:laurentexp} holds, because we will require absolute convergence, and because we will need to be able to bound certain quantities using values taken on its boundary. In particular, we define the annulus $A=\overline{\mathcal{B}\left(0,(R/r)^{1/4}\right)}\setminus\mathcal{B}\left(0,(r/R)^{1/4}\right)$. The radii here are chosen to scale with $R/r$. Note that taking the radii to be $(R/r)^s$ and $(r/R)^s$ for any $s\in(1,1/2)$ would work just as well, and this would slightly alter the error bound. We will return to this matter immediately after the proof of this theorem.

Now, since each $\omega^j\in A$ and $A$ is compact, we have
\begin{equation*}
\frac 1p\sum_{j=1}^pk_{x,y}\left(\omega^j\right)=\frac 1p\sum_{j=1}^p\sum_{l=-\infty}^\infty a_l\omega^{jl}=\frac 1p\sum_{l=-\infty}^\infty a_l\sum_{j=1}^p\omega^{jl}=\sum_{l=-\infty}^\infty a_{pl},
\end{equation*}
where the last line follows from the fact that $\sum_{j=1}^p\omega^{jl}$ is $p$ if $l$ is a multiple of $p$ and is $0$ otherwise. Hence, by \eqref{eq:kxyquad}, we get
\begin{equation}
|\varepsilon|=\left\vert a_0-\sum_{j=1}^p\frac{1}{p}k_{x,y}\left(\omega^j\right)\right\vert
=\left\vert a_0-\sum_{l=-\infty}^\infty a_{pl}\right\vert
\leq\sum_{l=-\infty}^{-1}\left\vert a_{pl}\right\vert+\sum_{l=1}^\infty\left\vert a_{pl}\right\vert.
\label{eq:laurentbound}
\end{equation}

Next, we bound the magnitude of the Laurent coefficient $a_{pl}$ of $k_{x,y}$ using $R$, $r$, $p$, and the maximum of $k$ over $F=\mathcal{B}(c,r\left(R/r\right)^{3/4})$. (Recall that $F$ is the shaded region in Figure~\ref{fig:proxy}.) To do so, define $\tilde{A}$ to be the image of $A$ under the change of coordinates $\tilde{z}=c+z\sqrt{Rr}$. That is, $\tilde{A}=\overline{F}\setminus\mathcal{B}\left(c,r(R/r)^{1/4}\right)$. Noting \eqref{eq:kxy}, we have%{\textcolor{red}{[Should the 2nd inequality below be an equality? Also, how about explicitly point out the mapping between $\tilde{z}\in F'$ and $z\in A$ through $\tilde{z}=c+z\sqrt{Rr}$?]}}
\begin{align*}
\max_{z\in A}|k_{x,y}(z)|&\leq\max_{z\in A}\left\vert\frac{z\sqrt{Rr}}{c+z\sqrt{Rr}-x}\right\vert\max_{z\in A}\left\vert k\left(c+z\sqrt{Rr},y\right)\right\vert\\
&=\max_{z\in \tilde{A}}\left|\frac{z-c}{z-x}\right|\max_{z\in \tilde{A}}|k(z,y)|\\
&=\frac{r(R/r)^{1/4}}{r((R/r)^{1/4}-1)}\max_{z\in \tilde{A}}|k(z,y)|=\frac{\alpha}{2}\max_{z\in \tilde{A}}|k(z,y)|\\
&\leq\frac{\alpha}{2}\max_{z\in F}|k(z,y)|\leq\frac{\alpha}{2}\max_{z\in\partial F}|k(z,y)|,
\end{align*}
where the last two inequalities follow from the maximum modulus principle and the fact that $k(z,y)$ is holomorphic on $z$ on all of $E$ (and not just $\tilde{A}$). %{\textcolor{red}{[how about moving $\le\frac{\alpha}{2}\max_{z\in F}\left\vert k(z,y)\right\vert $ from the following set of inequalities to above?]}}

Now, note that for $l>0$, we have the bound
\begin{equation}
\label{eq:firstlaurentbound}
|a_l|\leq\left\vert\frac{1}{2\pi}\int_{|\zeta|=(R/r)^{1/4}}\frac{k_{x,y}(\zeta)}{\zeta^{l+1}}d\zeta\right\vert\leq\frac{1}{2\pi}\left(\frac{R}{r}\right)^{-l/4}\max_{|\zeta|=(R/r)^{1/4}}|k_{x,y}(\zeta)|.
\end{equation}
Similarly, for $l<0$, we have the bound 
\begin{equation}
\label{eq:secondlaurentbound}
|a_l|\leq\left\vert\frac{1}{2\pi}\int_{|\zeta|=(r/R)^{1/4}}\frac{k_{x,y}(\zeta)}{\zeta^{l+1}}d\zeta\right\vert\leq\frac{1}{2\pi}\left(\frac{r}{R}\right)^{-l/4}\max_{|\zeta|=(r/R)^{1/4}}|k_{x,y}(\zeta)|.
\end{equation}
Hence, for each $l\in\mathbb{Z}$ with $l\neq0$, we have
\begin{align*}
\left\vert a_l\right\vert&\leq\max\left\{\left\vert\frac{1}{2\pi}\int_{\vert\zeta\vert=(R/r)^{1/4}}\frac{k_{x,y}(\zeta)}{\zeta^{l+1}}d\zeta\right\vert,\left\vert\frac{1}{2\pi}\int_{\vert\zeta\vert=(r/R)^{1/4}}\frac{k_{x,y}(\zeta)}{\zeta^{l+1}}d\zeta\right\vert\right\}\\
%{\textcolor{red}{[is\; this\; step\; correct? maybe\; explain]}}
&\leq\frac{\max_{z\in A}\left\vert k_{x,y}(z)\right\vert}{\left((R/r)^{\vert l\vert/4}\right)}
\leq\frac{\alpha}{2}\frac{\max_{z\in\partial F}\left\vert k(z,y)\right\vert}{\left((R/r)^{\vert l\vert/4}\right)}.
\end{align*}
Combining this with \eqref{eq:laurentbound}, we get
\[
|\varepsilon|\leq\frac{\alpha}{2}\left(2\sum_{l=1}^\infty\frac{\max_{z\in\partial F}\left\vert k(z,y)\right\vert}{\left((R/r)^{1/4}\right)^{pl}}\right)=\alpha\frac{\max_{z\in\partial F}\left\vert k(z,y)\right\vert}{(R/r)^{p/4}-1}.
\]
\end{proof}

A thorough discussion of related bounds can be found in \cite{trefethen}, along with an idea for an alternative proof using residue calculus. However, the bounds given there and elsewhere in the numerical analysis literature do not simultaneously and explicitly bound the quadrature error for all values of the enclosed set $X$ for each $y\in Y$, which is the setting of the proxy point method. As we mentioned, a word regarding the constant $s=1/4$ used in defining $A$ in the proof above is in order. This constant governs the coefficient of $p$ in the final asymptotic rate of convergence of Equation~\eqref{eq:termwisebound}. In fact, it is possible in some cases to carry out a more sophisticated analysis from that which is given here, as suggested in \cite{trefethen} but not explicitly carried out anywhere in the literature (to the authors' knowledge), using $s=1/2$ and the Cauchy principal value instead of the contour integral in \eqref{eq:firstlaurentbound} and \eqref{eq:secondlaurentbound}. This would improve asymptotic bounds by a factor of two, but would require assuming that $k$ only has first-order pole singularities.

Finally, we note that the setup of this proof also provides support for the heuristic, noted above and shown in \cite{kercompr} for the Cauchy kernel, that we should pick $C$ to have radius $\sqrt{Rr}$ as in the setup of this section.

We next use Proposition~\ref{prop:proxelwisebound} to bound the entrywise error for the low-rank approximation of a kernel matrix via the proxy point point method, which allows us to guarantee applicability of the HSS construction method in Sections~\ref{sec:leafhss} and \ref{sec:resthss}

\subsection{Accuracy for kernel matrix low-rank approximations}
The termwise error bound for each element $k(x,y)$ allows us to obtain an absolute 2-norm error bound for the approximation to the matrix $k(X,Y)$. Furthermore, if $k$ satisfies a univalence condition, we may obtain a relative 2-norm error bound that guarantees exponential convergence in the number of proxy points.

\begin{proposition1}
\label{prop:fnbound}
Suppose that the conditions in Proposition~\ref{prop:proxelwisebound} hold, and let $X=\{x_1,\ldots,x_l\}$, $Y=\{y_1,\ldots,y_m\}$, and $Z=\{z_1,\ldots,z_p\}$. Define
\begin{equation*}
\qquad\qquad\qquad U=\sqrt{Rr}\left(\frac{1}{z_j-x_i}\right)_{x_i\in X,z_j\in Z}
\operatorname{diag}(\omega,\omega^2,\ldots,\omega^p),\quad
%V&=\begin{pmatrix} k\left(y_1,z_1\right) & k\left(y_2,z_1\right) & \cdots & k\left(y_m,z_1\right)\\
%k\left(y_1,z_2\right) & k\left(y_2,z_2\right) & \cdots & k\left(y_m,z_2\right)\\
%\vdots & \vdots & \ddots & \vdots\\
%k\left(y_1,z_p\right) & k\left(y_2,z_p\right) & \cdots & k\left(y_m,z_p\right)\end{pmatrix},
V=k(Y,Z),
\end{equation*}
where $\operatorname{diag}()$ denotes a diagonal matrix. Then with the notation in Proposition~\ref{prop:proxelwisebound},
\begin{equation}
\left\Vert k(X,Y)-UV^T\right\Vert_2\leq lm\alpha\frac{\max_{y\in Y,z\in\partial F}|k(z,y)|}{(R/r)^{p/4}-1}.
\label{eq:proxtol}
\end{equation}

Furthermore, if in addition, $l\geq2$, $c$ is one of the points in $X$, and $k(z,y)$ is bounded and univalent as a function of $z$ on $E$ for each $y\in Y$, then
\begin{equation*}
\frac{\left\Vert k(X,Y)-UV^T\right\Vert_2}{\left\Vert k(X,Y)\right\Vert_2}\leq \frac{lm\alpha(1+\beta)}{(R/r)^{p/4}-1},
\end{equation*}
where $\beta=\frac{2(R/r)^{3/4}(1+(r/R)^2)}{(1-(r/R)^{1/4})^2}$.%{\textcolor{red}{since both $\alpha$ and $\beta$ are defined from $r,R$, would it be better to use a single symbol for a simplified representation of $\alpha\beta$?}}
\end{proposition1}
\begin{proof}
The first result is obvious based on the entrywise error bound in Proposition~\ref{prop:proxelwisebound}.
To see the second result, we first bound the function maximum on the right-hand side of \eqref{eq:proxtol}. The condition that $k(z,y)$ is univalent in $z$ on $E$ allows us to bound its growth away from $c$ by the distance from $c$. In particular, we use the growth theorem for univalent maps on the unit disk that take the value 0 and have derivative equal to 1 at the origin. Such maps are called regular univalent, or {\it schlicht}, maps. We define a regular map $g_y(z)$ on the unit disk that takes values related to $k(z,y)$, use the growth theorem to bound its growth away from 0, and then use this to bound the growth of $k(z,y)$ away from $c$. More precisely, for each $y\in Y$, we define $g_y:\mathbb{D}\to\mathbb{C}$ as follows:%{\textcolor{red}{[$\mathbb{D}$ is undefined. I suppose it is the unit disk?]}}
\begin{align*}
f_y(z)&=k(z,y),\\
h_y(z)&=Rz+c,\;\mathrm{and}\\
g_y(z)&=\frac{\left(f_y\circ h_y\right)(z)-\left(f_y\circ h_y\right)(0)}{\left(f_y\circ h_y\right)'(0)}.
\end{align*}
(Note that the denominator is nonzero because $f_y\circ h_y$ is a composition of the univalent maps $f_y$ and $h_y$ on $\mathcal{B}(c,R)$ and $\mathbb{D}$, respectively, and hence univalent.) Then each $g_y$ is regular and univalent, so by the growth theorem, we have $\frac{\vert z\vert}{\left(1+\vert z\vert\right)^2}\leq\left\vert g_y(z)\right\vert\leq\frac{\vert z\vert}{\left(1-\vert z\vert\right)^2}$. Thus, for $z\in\mathbb{D}$,
\begin{align}
\frac{\vert z\vert\left\vert\left(f_y\circ h_y\right)'(0)\right\vert}{\left(1+\vert z\vert\right)^2}&\leq\left\vert\left(f_y\circ h_y\right)(z)-\left(f_y\circ h_y\right)(0)\right\vert\leq\frac{\vert z\vert\left\vert\left(f_y\circ h_y\right)'(0)\right\vert}{\left(1-\vert z\vert\right)^2},\;\mathrm{or}\nonumber\\
\label{eq:growththm}
\frac{\vert z\vert R\left\vert f'_y(c)\right\vert}{\left(1+\vert z\vert\right)^2}&\leq\left\vert k(Rz+c,y)-k(c,y)\right\vert\leq\frac{\vert z\vert R\left\vert f'_y(c)\right\vert}{\left(1-\vert z\vert\right)^2}.
\end{align}
Therefore, from the second inequality of \eqref{eq:growththm}, we have for $z\in\partial\mathcal{B}\left(0,(r/R)^{1/4}\right)$ that
\begin{align*}
\left\vert k(Rz+c)\right\vert&\leq\frac{R(r/R)^{1/4}\left\vert f_y'(c)\right\vert}{(1-(r/R)^{1/4})^2}+\left\vert k(c,y)\right\vert\nonumber\\
&=\frac{R^{3/4}r^{1/4}\left\vert f_y'(c)\right\vert}{(1-(r/R)^{1/4})^2}+\left\vert k(c,y)\right\vert.
\end{align*}
Define $\gamma_1=\frac{R^{3/4}r^{2/4}}{(1-(r/R)^{1/4})^2}$. Then, for $z\in\partial F$ and $y\in Y$,
\begin{equation}
\label{eq:growththm2}
\left\vert k(z,y)\right\vert\leq\gamma_1\left\vert f_y'(c)\right\vert+\left\vert k(c,y)\right\vert.
\end{equation}
Now, defining $\gamma_2=\frac{\alpha}{(R/r)^{p/4}-1}$ and $y_0=\arg\max_{y\in Y}\left(\gamma_1\left\vert f_y'(c)\right\vert+\left\vert k(c,y)\right\vert\right)$, by Equations~\eqref{eq:proxtol} and \eqref{eq:growththm2}, it follows that
\begin{align}
\frac{\left\Vert k(X,Y)-UV^T\right\Vert_2}{\left\Vert k(X,Y)\right\Vert_2}&\leq lm\gamma_2\left(\frac{\gamma_1\left\vert f_{y_0}'(c)\right\vert+\left\vert k(c,y_0)\right\vert}{\left\Vert k(X,Y)\right\Vert_2}\right)\nonumber\\
&\leq lm\gamma_2\left(\frac{\gamma_1\left\vert f_{y_0}'(c)\right\vert+\left\vert k(c,y_0)\right\vert}{\left\Vert k(X,Y)\right\Vert_1}\right)\nonumber\\
&=lm\gamma_2\left(\frac{\gamma_1\left\vert f_{y_0}'(c)\right\vert+\left\vert k(c,y_0)\right\vert}{\max_{y\in Y}\left(\sum_{j=1}^l\left\vert k(x_j,y)\right\vert\right)}\right).\label{eq:relbound}
\end{align}
From the fact that $l\geq2$ and $c=x_{j_0}$, for some integer $j_0\in[1,l]$, there exists an integer $j_1\in[1,l]$ such that $x_{j_1}$ is a distance $r$ away from $x_{j_0}$. We know from the first inequality of \eqref{eq:growththm} that
\begin{equation*}
\frac{(r/R)R|f_y'(c)|}{1+(r/R)^2}=\frac{|(1/R)(x_{j_1}-c)|R|f_y'(c)|}{1+|(1/R)(x_{j_1}-c)|^2}\leq|k(x_{j_1},y)-k(x_{j_0},y)|.
\end{equation*}
Hence, by the triangle inequality applied to $k(x_{j_0},y)$, $k(x_{j_1},y)$, and $0$, we have
\begin{equation*}
(1/2)\frac{r\left\vert f_y'(c)\right\vert}{(1+(r/R)^2)}\leq\max(\left\vert k(x_{j_1},y)\right\vert,\left\vert k(x_{j_0},y)\right\vert)=\max(\left\vert k(x_{j_1},y)\right\vert,\left\vert k(c,y)\right\vert).
\end{equation*}
In particular, we then have the following bound on the denominator of $\eqref{eq:relbound}$:
\begin{align}
\max\left(\left\vert k(c,y_0)\right\vert,\frac{r\left\vert f_{y_0}'(c)\right\vert}{2(1+(r/R)^2)}\right)&\leq\max_{y\in Y}\left(\max\left(\left\vert k(c,y)\right\vert,\frac{r\left\vert f_y'(c)\right\vert}{2(1+(r/R)^2)}\right)\right)\nonumber\\
&\leq\max_{y\in Y}\left(\sum_{j=1}^l\left\vert k(x_j,y)\right\vert\right).
\label{eq:trianglelbdenom}
\end{align}
Combining \eqref{eq:relbound} and \eqref{eq:trianglelbdenom}, we thus have
\begin{align*}
\frac{\left\Vert k(X,Y)-UV^T\right\Vert_2}{\left\Vert k(X,Y)\right\Vert_2}&\leq lm\gamma_2\left(\frac{\gamma_1\left\vert f_{y_0}'(c)\right\vert+\left\vert k(c,y_0)\right\vert}{\max\left(\left\vert k(c,y_0)\right\vert,\frac{r\left\vert f_{y_0}'(c)\right\vert}{2(1+(r/R)^2)}\right)}\right)\\
&\leq lm\gamma_2\left(\frac{\left\vert f_{y_0}'(c)\right\vert}{\left\vert f_{y_0}'(c)\right\vert}+\frac{\frac{R^{3/4}r^{1/4}\left\vert f_{y_0}'(c)\right\vert}{(1-(r/R)^{1/4})^2}}{\frac{r\left\vert f_{y_0}'(c)\right\vert}{2(1+(r/R)^2)}}\right)\\
&=lm\gamma_2(1+\beta).
\end{align*}
%{\textcolor{red}{how did $\alpha$ reappear above?}}
The result then follows by our definition of $\gamma_2$.
\end{proof}

Hence, for a given 2-norm tolerance $\tau$ of the proxy point approximation to $k(X,Y)$, we only need to use $O\left(\log lm+\log\tau\right)$ proxy points, as long as the assumption on the analyticity of $k$ holds, and as long as $k$ grows sub-exponentially on the relevant domains. If $k$ satisfies a certain univalence condition, then Proposition~\ref{prop:fnbound} guarantees slow growth and hence a relative error bound for the approximation to $k(X,Y)$ that decreases exponentially in the number of proxy points $p$. It is worth noting again that the ratio $R/r$, determined by the geometry of the points $X$ and $Y$ and the analyticity of $k$, continues to be the key parameter governing convergence of the proxy point approximation to $k(X,Y)$.

\subsection{Approximating Toeplitz matrices}
\label{sec:toeplitzapprox}
To illustrate how the above analysis may be useful, we consider one specific application: approximating a Toeplitz matrix with Toeplitz vector generated by a univalent function. In the next section, we then use this to design a quick rank-structured approximation method for such Toeplitz matrices.

Consider an $n\times n$ real- or complex-valued Toeplitz matrix
\begin{equation}
T=\begin{pmatrix}t_0 & t_{-1} & \ldots & t_{-(n-1)}\\
t_1 & t_0 & \ldots & t_{-(n-2)}\\
\vdots & \vdots & \ddots & \vdots\\
t_{n-1} & t_{n-2} & \ldots & t_0 \end{pmatrix},
\end{equation}
where $t_i=f_1(i)$ for $-n\leq i\leq-1$ and $t_i=f_2(i)$ for $1\leq i\leq n$, with univalent functions $f_1\in\mathcal{O}\left(\mathcal{B}\left(-n/2,n/2\right)\right)$ and $f_2\in\mathcal{O}\left(\mathcal{B}\left(n/2,n/2\right)\right)$. Such matrices arise in \cite{bmb,wws}, as well as in the Gaussian process literature mentioned in the introduction \cite{gaussproc}. In \cite{bmb}, for example, $f_1,f_2$ are defined by $f_2(z)=(1-z)\log\left((z-1)/z\right)+(z+1)\log\left(z/(z+1)\right)$ and $f_1(z)=-f_2(z)$. Other commonly-used kernels that are univalent on the relevant domain include the Cauchy kernel and the Gaussian kernel with large (in this context, $O(n)$) length scale.

From now on, we assume without loss of generality that $n$ is a multiple of 8. This is purely for convenience of notation and is not a restriction on the applicability of the main ideas. We may consider, for example, a given off-diagonal block of $T$ to be a kernel matrix corresponding to the kernel $k(x,y)=f_2(x-y)$:
\begin{equation*}
k(X,Y)=T_{[n/2+n/4+1,n-n/4]\times[1,n/2]},
\end{equation*}
where $X=[n/2+n/4+1,n-n/4]$ and $Y=[1,n/2]$. By our assumption on $f_2$, we are able to use the proxy point method with center $3n/4$ and radius $n/(4\sqrt{2})$ to get an approximation for $k(X,Y)$. Note that here, $R=n/2$ and $r=n/4$, so $R/r=2$. Ensuring this separation between $X$ and $Y$, and hence the analyticity of $f_2$, is the reason why we did not pick $X=[n/2+1,n]$. More generally, this consideration is why using the proxy point method may often not work well to approximate the entire bottom-left subblock of $T$ in this manner. Using Equation~\eqref{eq:proxtol}, together with the function bound in Proposition~\ref{prop:toepuniv} below, guarantees that we would need $O(\log n)$ proxy points to get a given approximation accuracy for large $n$.

\begin{proposition1}
\label{prop:toepuniv}
Let $f$ be holomorphic, bounded, and univalent on $\mathcal{B}\left(n/2,n/2\right)$. Then for $z\in\mathcal{B}\left((n+1)/2,n/2-1\right)$,
\begin{equation*}
\left\vert f(z)\right\vert\leq\left(\frac{n}{2}\right)^3\left\vert f'\left(\frac{n}{2}\right)\right\vert+\left\vert f\left(\frac{n}{2}\right)\right\vert.
\end{equation*}
\end{proposition1}
\begin{proof}
This is an adaptation of the proof of Proposition~\ref{prop:fnbound}; we modify it here to explicitly relate $n$ to the growth rate of $f$. As before, define $g:\mathbb{D}\to\mathbb{C}$ in the following way:
\begin{align*}
h(z)=\left(n/2\right)z+n/2,\quad g(z)=\frac{\left(f\circ h\right)(z)-\left(f\circ h\right)(0)}{\left(f\circ h\right)'(0)}.
\end{align*}
Then $g$ is a regular univalent function, so we have $\left\vert g(z)\right\vert\leq\frac{\vert z\vert}{\left(1-\vert z\vert\right)^2}$ by the growth theorem. Thus, for $z\in\mathcal{B}\left(0,(n-2)/n\right)$,
\begin{equation*}
\left\vert\left(f\circ h\right)(z)-\left(f\circ h\right)(0)\right\vert\leq\frac{\vert z\vert\left\vert\left(f\circ h\right)'(0)\right\vert}{\left(1-\vert z\vert\right)^2}.
\end{equation*}
Therefore, for $z\in\mathcal{B}((n+1)/2,n/2-1)\subseteq\mathcal{B}(n/2,n/2)$ (see Figure~\ref{fig:univregion} for an illustration of these two regions), we have
\begin{align*}
\left\vert f(z)\right\vert=\left\vert\left(f\circ h\right)((2/n)(z-n/2))\right\vert&\leq\frac{\vert (2/n)(z-n/2)\vert\left\vert\left(f\circ h\right)'(0)\right\vert}{\left(1-\vert(2/n)(z-n/2)\vert\right)^2}+\left\vert\left(f\circ h\right)(0)\right\vert\\
&\leq\frac{(1-2/n)|(f\circ h)'(0)|}{(1-(n-2)/n)^2}+|(f\circ h)(0)|\\
&\leq\left(\frac{n}{2}\right)^2\left\vert\left(f\circ h\right)'(0)\right\vert+\left\vert\left(f\circ h\right)(0)\right\vert\\
%&=\left(n/2\right)^2\left\vert f'\left(n/2\right)h'(0)\right\vert+\left\vert f\left(n/2\right)\right\vert\\
&\leq\left(\frac{n}{2}\right)^2\left\vert f'\left(\frac{n}{2}\right)\right\vert\left\vert h'(0)\right\vert+\left\vert f\left(n/2\right)\right\vert\\
&=\left(\frac{n}{2}\right)^2\left(\frac{n}{2}\right)\left\vert f'\left(\frac{n}{2}\right)\right\vert+\left\vert f\left(\frac{n}{2}\right)\right\vert,
\end{align*}
so the result follows from the definition of $h$.
\end{proof}

\begin{figure}[ptbh]
\centering
\tabcolsep5mm
\includegraphics[height=1.8in]{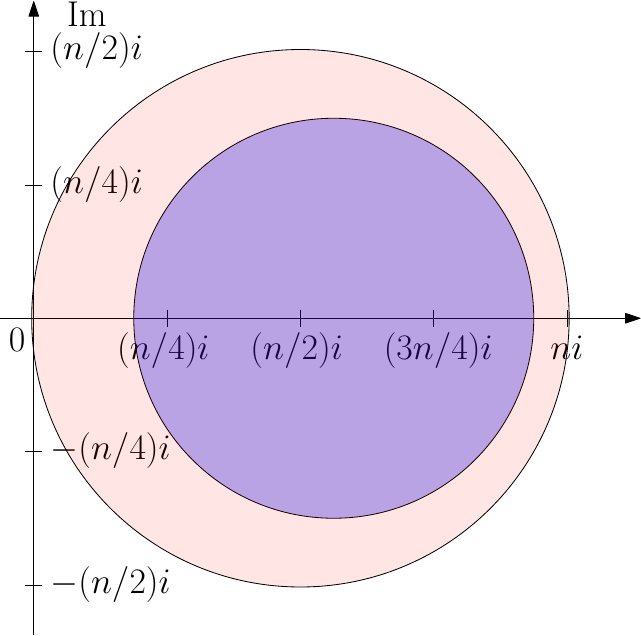}
\caption{The region $\mathcal{B}((n+1)/2,n/2-1)$ (blue), used in the proof of Proposition~\ref{prop:toepuniv}, inside the region $\mathcal{B}(n/2,n/2)$ (red) on which $f$ is assumed to be univalent.}
\label{fig:univregion}
\end{figure}

Plugging the above bound into \eqref{eq:proxtol} and using the fact that, in this case, $R/r=2$, we get the following bound for the 2-norm error incurred using a proxy point approximation $k(X,Y)\approx UV$ with $p$ points:
\begin{equation*}
\left\Vert k(X,Y)-UV\right\Vert_2\leq\left(n^2/8\right)\left(\frac{2^{5/4}}{2^{1/4}-1}\right)\left(\frac{1}{2^{p/4}-1}\right)\left(\left(n/2\right)^3\left\vert f'(n/2)\right\vert+\left\vert f(n/2)\right\vert\right).
\end{equation*}
Therefore, a given error tolerance requires $O(\log n)$ proxy points. In the following section, we further develop this idea to construct an HSS approximation to $T$ with a computational cost that is sublinear in $n$.

\section{Sublinear-complexity HSS approximations of some Toeplitz matrices}
\label{sec:toeplitzconst}
In this section, we detail our sublinear HSS construction algorithm for Toeplitz matrices arising from univalent maps applied to a regular grid in one dimension. The combination of ideas necessary for this method was first explored in \cite{hypercauchy}. More recently, similar ideas have been explored in \cite{bkw}.%The approximation construction algorithm is detailed in Sections~\ref{sec:leafhss} and \ref{sec:resthss}. The analysis, started in Section~\ref{sec:proxy}, of the number of proxy points necessary for a good approximation is continued in this context in Section~\ref{sec:numberproxy}.

To understand the utility of the new scheme, it is worth briefly reviewing existing Toeplitz methods. Over the past six decades, many algorithms have been devised that exploit the additional structure of Toeplitz matrices to perform various matrix operations faster than the counterpart ``naive" algorithms applicable to general matrices. For example, so-called ``fast" (faster than cubic time in the size of the matrix) and ``superfast" (faster than quadratic time in the size of the matrix) algorithms have been devised to solve Toeplitz systems \cite{hei95,goh95,gu98}. The central idea of such algorithms over the past few decades has become to apply fast Fourier transforms (FFTs) and solve the equivalent system in the frequency space. The resulting Cauchy-like matrix turns out both to be quickly solved by Gaussian elimination and to have low off-diagonal rank; hence, it can be quickly approximated by structured matrices \cite{goh95,mar05,toep}. Similarly, in digital signal processing, it has become well-known that the multiplication of Toeplitz convolution matrices with a given signal can be accelerated by applying FFTs and performing the equivalent operation in the frequency domain \cite{jain,burger}.

After certain speedups that may be obtained using randomized techniques, the dominant cost in such structured matrix frequency-domain Toeplitz solution and multiplication algorithms becomes the application of FFTs \cite{fasthss,lib07,toeprs,bkw}. Hence, in theory, general HSS algorithms can potentially achieve a speedup for matrix operations whenever a matrix is both Toeplitz and has low off-diagonal rank before the application of FFTs \cite{fasthss}. In such algorithms, the dominant cost becomes the construction of the structured approximant; thus, bringing this cost down is a worthwhile endeavor. In this work, we show that for Toeplitz matrices whose Toeplitz vector is generated by a univalent map applied to the positive integers, we are able to reduce the HSS construction time cost from $O\left(r^2n\right)$ \cite{toepls,hsscost} or $O\left(rn\right)$ \cite{bkw} to $O\left(\log^5n\right)$ in the size $n$ of a square matrix with off-diagonal rank bound $r$. While the new algorithm is less widely applicable, it may nevertheless be applied to certain important classes of matrices, such as those arising as covariance matrices of Gaussian processes \cite{gaussproc,gausstoep}, or from a convolution of a digital signal with a large Gaussian filter \cite{burger}. In addition, since this new scheme does not rely on Fourier space representation, it has the advantage of preserving the rank structure of any diagonal or rank-structured summand that may be added to the Toeplitz matrix, such as when localizing eigenvalues \cite{eig14,vog16}.%{\textcolor{red}{[since section 2.3 already had some discussions of the motivation for studying $T$, I suggest to move some discussions from here to there]}}

The first key idea in our new construction scheme is the use of the proxy point method in the process of obtaining an interpolative decomposition (also known as skeletonization) of the HSS blocks, as was done previously in \cite{mar05-1}. The second key idea is the reuse of the resulting approximate basis matrix factors for all the HSS blocks at a given HSS depth, as was done previously in \cite{hypercauchy}. Here is where we use our new analysis from Section~\ref{sec:proxy} to guide the process of obtaining these approximate basis factors, as well as to understand when the construction scheme is applicable. In the case that the proxy point method is used to approximate off-diagonal blocks of Toeplitz matrices with Toeplitz vector generated by a complex-analytic univalent map, this error is then shown to increase slowly enough in $n$ to allow our construction algorithm to be performed in sublinear time relative to $n$. While we do not perform an operation count to justify this here, since our algorithm is almost identical to the one outlined in \cite{hypercauchy}, the analysis from Section 5 of that paper applies to the algorithm outlined in this section.

\subsection{Review of HSS matrix approximations}
\label{sec:review}
Next, we review the data structure known as a hierarchically semiseparable (HSS) matrix form. Here we only give a brief outline to introduce notation; more details can be found, e.g., in \cite{fasthss}.%{\textcolor{red}{[since HSS is now well known, please condense the following as much as possible; there is no need to give a formal definition; for figures 3.1 and 3.2, at most one of them would be enough]}}
\begin{definition}\label{def:hss}Let $M$ be a matrix. Assume without loss of generality that $M$ is square with row/column size $n$ equal to a power of two, and let $L<\log_2(n)$. Recursively partition in two the set of row/column indices of $M$ for a total of $2^L-1$ subsets. Specifically, for each $0\leq l\leq L$, partition $[1:n]$ into the $2^l$ sets
\begin{equation*}
\mathcal{I}_l=\left\{\left[1:\frac{n}{2^l}\right],\left[\frac{n}{2^l}+1:\frac{n}{2^{l-1}}\right],\ldots,\left[(2^l-1)\frac{n}{2^l}+1:n\right]\right\}.
\end{equation*}
Let $\mathcal{I}=\bigcup_{j=0}^{L}\mathcal{I}_l$, and impose a partial order on $\mathcal{I}$ by set inclusion. We call $\mathcal{I}$ the \emph{$L$-level HSS index set} of $M$. Then its Hasse diagram $\mathcal{T}$ is a perfect binary tree, called the \emph{HSS tree} of $M$. Now, for each $1\leq j\leq 2^L-1$, define ${\bf i}_j\in\mathcal{I}$ to be the element corresponding to the $j$th vertex of $\mathcal{T}$ in its postordered traversal. For each $1\leq j\leq 2^L-1$, define $M^-_j=M_{{\bf i}_j\times[1:n]\setminus{\bf i}_j}$ and $M^|_j=M_{[1:n]\setminus{\bf i}_j\times{\bf i}_j}$; these are called the \emph{$j$th HSS block row} and \emph{$j$th HSS block column}, respectively. The \emph{HSS rank} of $M$ is the maximum rank, over all $1\leq j\leq 2^L-1$, of $M^-_j$ and $M^|_j$.

%\begin{figure}[ptbh]
%\centering
%\tabcolsep5mm
%\begin{tabular}[c]{p{2.1in}p{2.1in}}
%\centering
%\includegraphics[height=2in]{hssblockrow} & \includegraphics[height=2in]{hssblockcol}\\
%\vspace{.1in}{\qquad\quad\enspace\small (a) HSS block rows} & \vspace{.1in}{\qquad\enspace\small (b) HSS block columns}
%\end{tabular}
%\caption{The HSS block rows and columns of $M$ where $L=2$. The labeled green blocks with rounded corners correspond to the HSS tree depth $l=1$; the labeled yellow blocks with sharp corners correspond to the HSS tree depth $l=2$.}\label{fig:hssblocks}
%\end{figure}

An \emph{$L$-level HSS form} for $M$ is a 6-tuple $\{{\bf D},{\bf U},{\bf R},{\bf V},{\bf W},{\bf B}\}$, where:
\begin{itemize}
\item ${\bf U}=\{U_j\}_{1\leq j\leq2^L-2}$, ${\bf V}=\{V_j\}_{1\leq j\leq2^L-2}$, and ${\bf B}=\{B_j\}_{1\leq j\leq2^L-2}$ are sets of matrices;
\item ${\bf D}=\{D_j\}_{j\in{\bf I}}$ is a set of matrices, where ${\bf I}$ is the set of postordered indices of leaves of $\mathcal{T}$;
\item and ${\bf R}=\{R_j\}_{j\in{\bf J}}$ and ${\bf W}=\{W_j\}_{j\in{\bf J}}$ are sets of matrices, where ${\bf J}$ is the set of postordered indices of vertices of $\mathcal{T}$ of depth at least two;
\end{itemize}
such that
\begin{enumerate}
\item $D_j=M_{{\bf i}_j\times{\bf i}_j}$ for $j\in{\bf I}$;
\item $M_{{\bf i}_j\times{\bf i}_{\mathrm{sib}(j)}}=U_jB_jV_{\mathrm{sib}(j)}^T$ for $1\leq j\leq2^L-2$, where $B_j$ is full-rank and $\mathrm{sib}(j)$ is the postordered index of the sibling of $j$;\label{item:hssapprox}
\item and $U_j=\begin{pmatrix}U_{c_1(j)}R_{c_1(j)} \\
U_{c_2(j)}R_{c_2(j)}\end{pmatrix}$ and $V_j=\begin{pmatrix}V_{c_1(j)}W_{c_1(j)} \\
V_{c_2(j)}W_{c_2(j)}\end{pmatrix}$ for $1\leq j\leq2^L-2$, where $c_1(j)$ and $c_2(j)$ denote the postordered indices of the left and right children of the postordered $j$th vertex of $\mathcal{T}$, respectively.
\end{enumerate}
\end{definition}
Collectively, all of the matrices mentioned in this definition are called \emph{HSS generators} of $M$. Note that we can find generators whose sizes can all be bounded by the HSS rank of $M$ \cite{fasthss}; this is the main point constructing the HSS form of $M$ and the reason for the efficiency of HSS algorithms. Figure~\ref{fig:hssgens} illustrates the various relationships of the HSS generators of $M$.

\begin{figure}[ptbh]
\centering
\tabcolsep5mm
\begin{tabular}[c]{p{2.1in}p{2.1in}}
\centering
\includegraphics[height=2in]{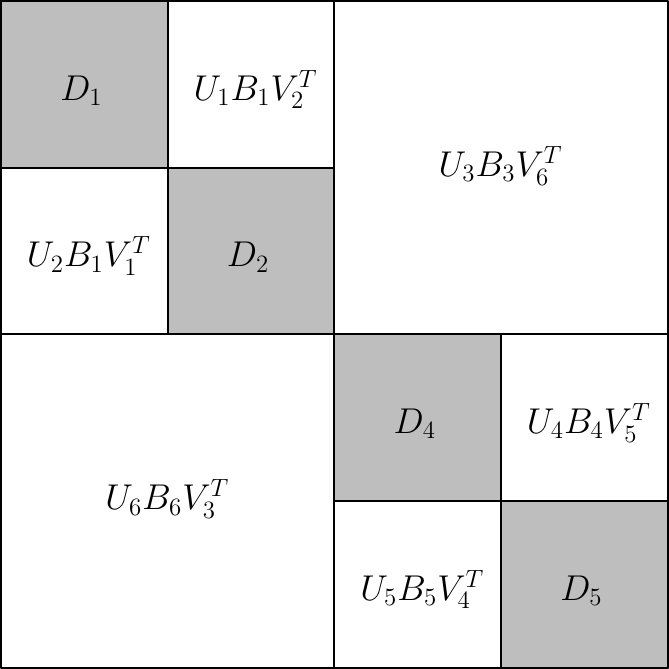} & \includegraphics[height=2in]{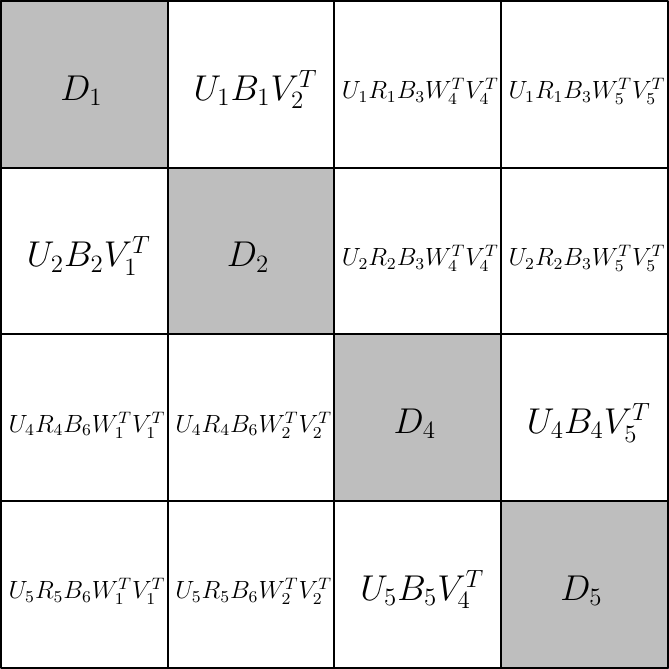}
\end{tabular}
\caption{The HSS generator products of $M$ placed into the blocks of $M$ that they generate.}\label{fig:hssgens}
\end{figure}

Finally, we say $M$ has \emph{numerical HSS rank $k$ with respect to a tolerance $\tau$} if the numerical rank of $M^-_j$ and $M^|_j$ with respect to a tolerance $\tau$ is at most $k$ over all $1\leq j\leq2^L-1$. We define an \emph{$L$-level rank-$k$ HSS approximation} of $M$ to be an $L$-level HSS form of $M$ where we replace condition~\ref{item:hssapprox} in Definition~\ref{def:hss} above with the following:
\begin{enumerate}
\item[$2'$] $M_{{\bf i}_j\times{\bf i}_{\mathrm{sib}(j)}}\approx U_jB_jV_{\mathrm{sib}(j)}^T$ for $1\leq j\leq2^L-2$, where $B_j$ is a $k\times k$ matrix.
\end{enumerate}

In the case that $M$ is a Toeplitz matrix, for $1\leq j\leq2^L-2$, existing methods of constructing any one of $U_j$, $V_j$, $R_j$, $W_j$, or $B_j$ in general scale linearly in $n$, for at least some $j$ \cite{toepls}. In Section~\ref{sec:proxyapprox}, we outline an algorithm to construct any such generator with sublinear cost. This is useful depending on how the HSS form of $M$ is subsequently used. For example, our method confers a speedup if only part of the output of a matrix-vector multiplication with $M$ is needed.

\subsection{HSS approximation of $T$ via the proxy point method}
\label{sec:proxyapprox}
%{\textcolor{red}{[since $T$ was already introduced in section 2.3, I suggest to reorganize the following introduction of $T$ and move to section 2.3 to avoid duplication]}}
Define $T$, $f_1$, and $f_2$ as in Section~\ref{sec:toeplitzapprox}. To more easily illustrate the application of this method, we will deal with the symmetric case $t_{-i}=t_i$ (so $f_1(-i)=f_2(i)$) for $i=0,\ldots,n-1$; define $f(z)=f_1(-z)$. The non-symmetric case is handled similarly (see Section~\ref{sec:resthss}). Since we are constructing generators for approximations to the off-diagonal blocks of $T$, we may again assume without loss of generality that $t_0=0$. Furthermore, since this algorithm is meant to apply to large matrices, we may assume that $n$ is a power of two greater than 8.

\subsubsection{Constructing the HSS row generators}
\label{sec:leafhss}
Let $L\leq\log_2(n)-2$ be the number levels in the desired HSS approximation to $T$. Let $r$ be a bound for the numerical HSS rank of $T$; we assume specifically that $r$ is $O(\log n)$. The analysis in Section~\ref{sec:numberproxy} can actually be used to give a bound for $r$. In particular, we can show that $r$ is $O(\log^2 n)$; see Section~\ref{sec:ext}.

For each $1\leq i,j\leq n$ with $i\neq j$, we have $T_{i,j}=f\left(|j - i|\right)$. Hence, we may consider an HSS block $T^-_j$ to be the kernel matrix $k({\bf i}_j,[1:n]\setminus{\bf i}_j)$, where $k$ is defined by $k(x,y)=f(|x-y|)$. Directly finding a low-rank factorization for $T^-_j$, for example as when $j=1$ in the first step in the HSS construction algorithm in \cite{fasthss}, is already prohibitively expensive with at least $O(n)$ flops. Instead, we may follow a similar list of steps as in \cite[Section~3.2]{hypercauchy}:

\begin{itemize}
\item If $j$ is not leaf of $\mathcal{T}$, we assume we have performed this list of steps on its children $c_1(j)$ and $c_2(j)$ to obtain sets of indices ${\bf i}'_{c_1(j)},{\bf i}'_{c_2(j)}\subseteq{\bf i}_j$. If $j$ is a leaf, we define $c_1(j)=c_2(j)=j$ and ${\bf i}'_j={\bf i}_j$. Then, we define ${\bf\overline{i}}'_j={\bf i}'_{c_1(j)}\cup{\bf i}'_{c_2(j)}$ and apply a proxy point approximation to $\left(T^-_j\right)_{{\bf\overline{i}}'_j\times[1:n-|{\bf i}_j|]}$. However, since we only assumed that $f$ is analytic on $\mathcal{B}\left(n/2,n/2\right)$, by Equation~\eqref{eq:proxtol}, the ratio $R/r$ in this case could be as large $1/n$, and therefore the number of proxy points $p$ required to obtain a reasonably good approximation may be prohibitively large. Hence, we first separate ${\bf i}_j$ into the ``near-field" and ``far-field" subsets ${\bf \hat{i}}_j$ and ${\bf \tilde{i}}_j={\bf i}_j\setminus{\bf \hat{i}}_j$, respectively, where ${\bf \hat{i}}_j$ is the subset of ${\bf i}_j$ consisting of its first and last $|{\bf i}_j|/4$ values, respectively, ordered the usual way. We then define ${\bf\hat{i}}'_j={\bf\hat{i}}_j\cap{\bf\overline{i}}'_j$, ${\bf\tilde{i}}'_j={\bf\tilde{i}}_j\cap{\bf\overline{i}}'_j$, $T^-_{j,1}=k\left({\bf\hat{i}}'_j,[1:n]\setminus{\bf i}_j\right)$, and $T^-_{j,2}=k\left({\bf\tilde{i}}'_j,[1:n]\setminus{\bf i}_j\right)$; and we apply a proxy point approximation to only the far-field subblock: $T^-_{j,2}\approx \tilde{U}_j\tilde{V}_j$. For this approximation, we use a circular contour with center $(1/2)\left(\min({\bf i}_j)+\max({\bf i}_j)\right)$ and radius $(\sqrt{2}/2)\left(\max({\bf i}_j)-\min({\bf i}_j)+1\right)$ to obtain $R/r=2$. (See Figure~\ref{fig:proxysetup} and Figure~\ref{fig:proxysetup2}.)

\begin{figure}[ptbh]
\centering
\includegraphics[height=2.03in]{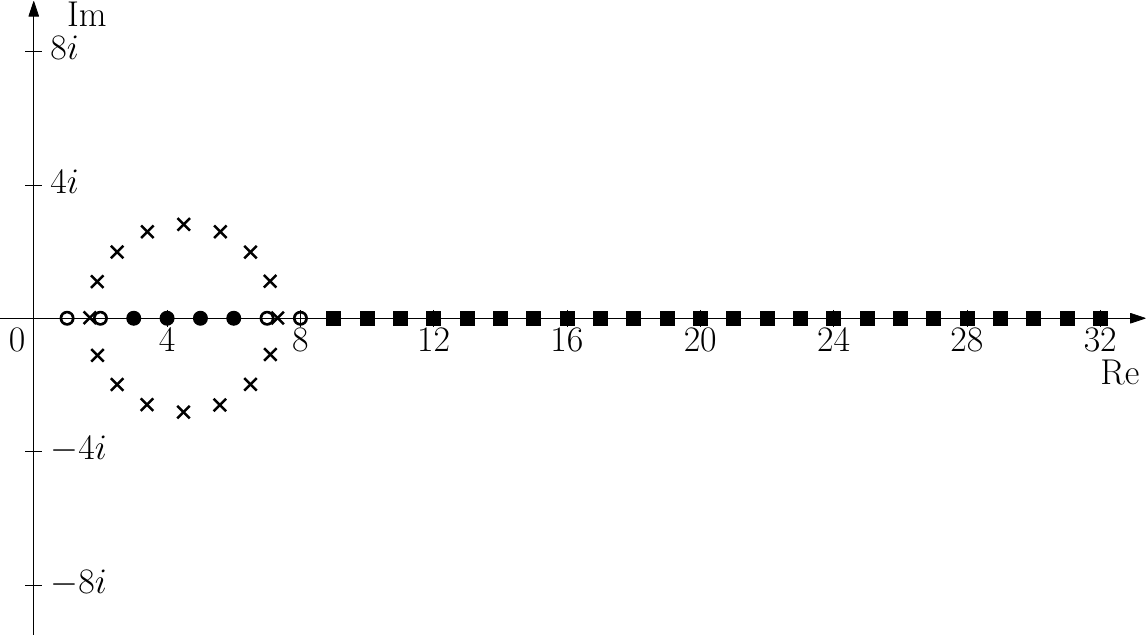}
\includegraphics[height=2in]{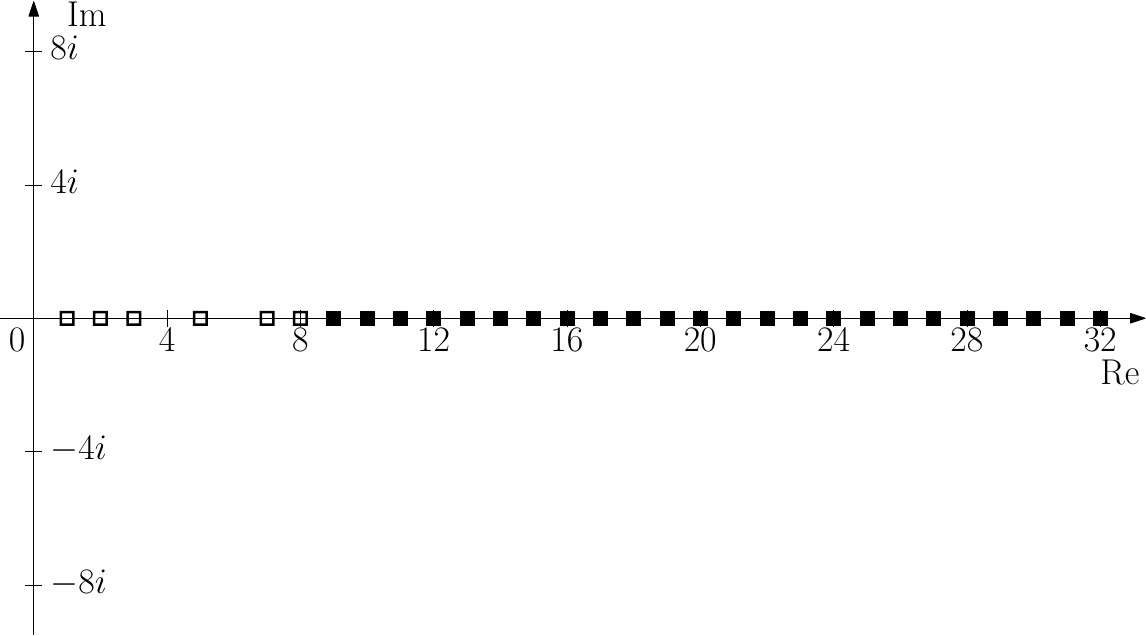}
\caption{Top: near-field points ${\bf\tilde{i}}'_1$ ({\protect\tiny $\circ$}), far-field points ${\bf\hat{i}}'_1$ ({\protect\tiny $\bullet$}), proxy points ({\protect\tiny $\times$}), and the points $[9:32]$ ({\protect\tiny $\blacksquare$}) involved in the approximation of the leaf HSS block $T^-_1\vert_{{\bf\overline{i}}'_1\times[1:n-|{\bf i}_1|]}=k\left([1:8],[9:32]\right)$ for a matrix of size $n=32$, number of HSS levels $L=2$, and number of proxy points $p=16$. Bottom: the resulting index set ${\bf i}'_1$ ({\protect\tiny $\square$}). (These are ``cartoon illustrations" and are not actual results from such an approximation applied to a subblock of an actual matrix $T$.)}
\label{fig:proxysetup}
\end{figure}
\begin{figure}[ptbh]
\centering
\includegraphics[height=2.03in]{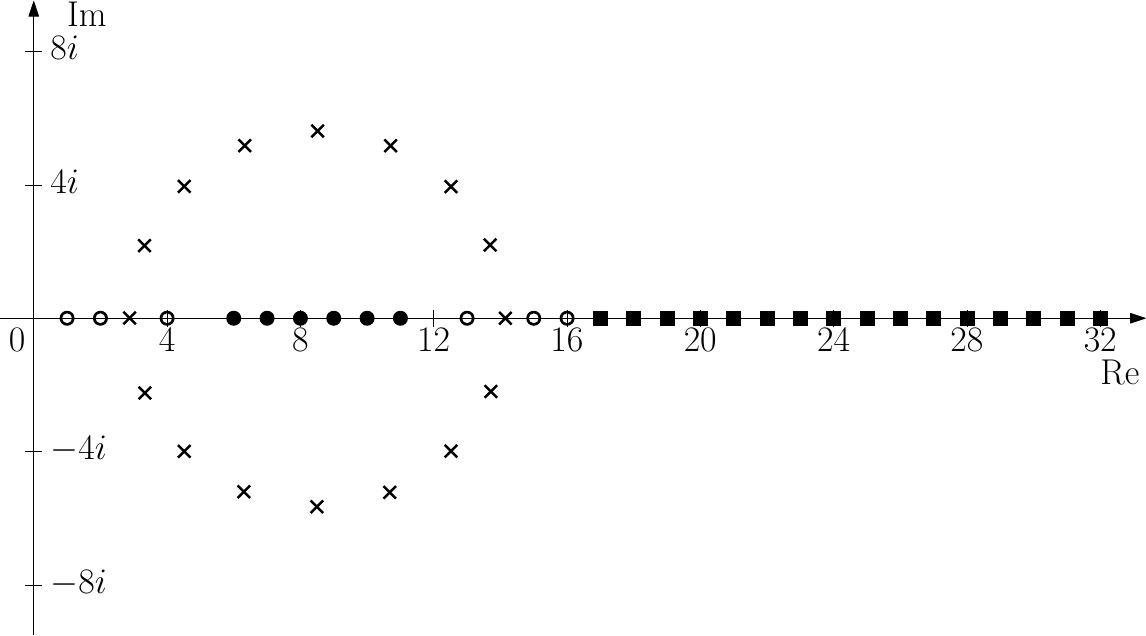}
\includegraphics[height=2in]{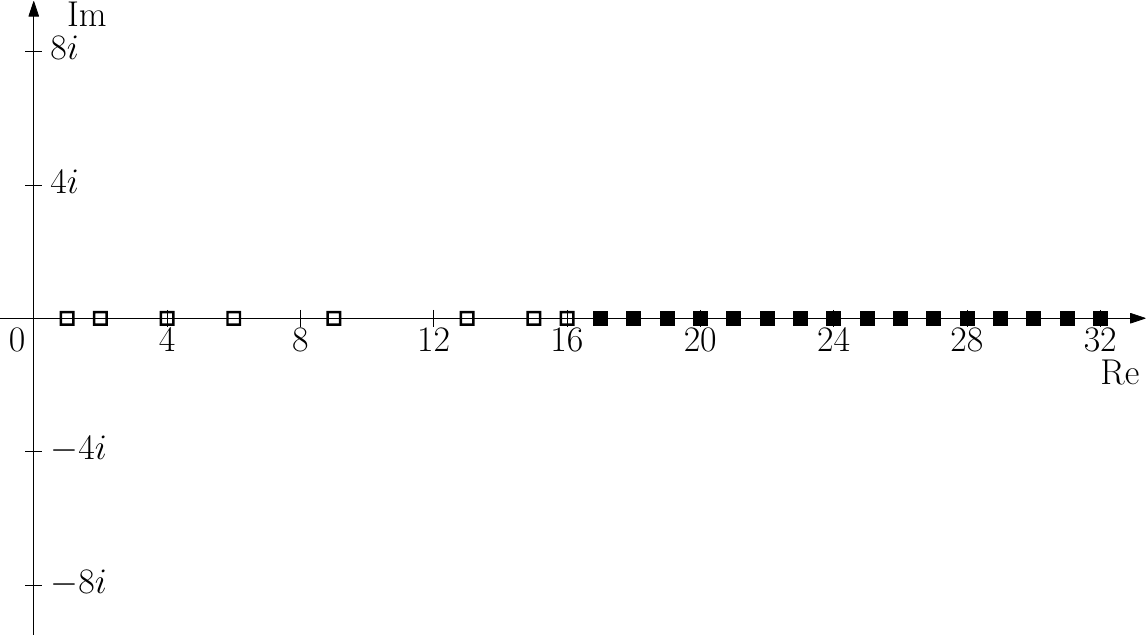}
\caption{Top: near-field points ${\bf\tilde{i}}'_3$ ({\protect\tiny $\circ$}), far-field points ${\bf\hat{i}}'_3$ ({\protect\tiny $\bullet$}), proxy points ({\protect\tiny $\times$}), and the points $[17:32]$ ({\protect\tiny $\blacksquare$}) involved in the approximation of height-2 HSS block $T^-_3\vert_{{\bf\overline{i}}'_3\times[1:n-|{\bf i}_3|]}=k({\bf\overline{i}}'_j,[17:32])$ for a matrix of size $n=32$, number of HSS levels $L=2$, and number of proxy points $p=16$. Bottom: the resulting index set ${\bf i}'_3$ ({\protect\tiny $\square$}). (As noted in Figure~\ref{fig:proxysetup} above, these are ``cartoon illustrations" and are not reflective of actual numerical results.)}
\label{fig:proxysetup2}
\end{figure}

We thus have
\begin{equation*}
\left(T^-_j\right)\vert_{{\bf\overline{i}}'_j\times[1:n-|{\bf i}_j|]}=\Pi_i\begin{pmatrix}
T^-_{j,1} \\
T^-_{j,2}
\end{pmatrix}=\Pi_i\begin{pmatrix}
I & 0 \\
0 & \tilde{U}_i\end{pmatrix}\begin{pmatrix}
T^-_{j,1} \\
\tilde{V}_i
\end{pmatrix},
\end{equation*}
where $\Pi_i$ is a permutation matrix.

\item Next, we find a strong rank-revealing QR factorization
\begin{equation*}
\tilde{U}_j=\overline{U}_j\left(\Pi^{\prime T}_j\tilde{U}_j\right)\vert_{[1:r]\times[1:p]},
\end{equation*}
where $\overline{U}_j=\left(\begin{array}[c]{cc}I & E_j\end{array}\right)^T$ and $\Pi'_j$ is a permutation matrix. In theory, any rank-revealing QR factorization may suffice, but in practice the SRRQR factorization results in greater numerical stability when working with $E_j$ (and hence $U_j$); see \cite{gu96} for details. We then have
\begin{equation*}
T^-_{j,2}\approx\overline{U}_j\left(\Pi^{\prime T}_j\tilde{U}_j\right)_{[1:r]\times[1:p]}\tilde{V}_j\approx\overline{U}_j\left(\Pi^{\prime T}_jT^-_{j,2}\right)_{[1:r]\times[1:n]\setminus{\bf i}_j},
\end{equation*}
so
\begin{align*}
\left(T^-_j\right)\vert_{{\bf\overline{i}}'_j\times[1:n-|{\bf i}_j|]}&\approx\Pi_j\begin{pmatrix}
T^-_{j,1} \\
T^-_{j,2}
\end{pmatrix}\\
&\approx U_j\begin{pmatrix}
\left(\Pi_j^TT^-_j|_{{\bf\overline{i}}'_j\times[1:n]\setminus{\bf i}_j}\right)|_{\left[1:\left\vert{\bf\hat{i}}'_j\right\vert\right]\times[1:n-|{\bf i}_j|]} \\
\left(\Pi_j^TT^-_j|_{{\bf\overline{i}}'_j\times[1:n]\setminus{\bf i}_j}\right)|_{\left[\left\vert{\bf\hat{i}}'_j\right\vert+1:\left\vert{\bf\hat{i}}'_j\right\vert+r\right]\times[1:n-|{\bf i}_j|]}
\end{pmatrix}\\
&=U_jT^-\vert_{{\bf i}'_j\times[1,n]\setminus{\bf i}_j},
\end{align*}
where ${\bf i}'_j\subseteq{\bf i}_j$ is of size $\left\vert{\bf\hat{i}}'_j\right\vert+r$ and
\begin{equation*}
U_j=\Pi_j\begin{pmatrix}
I & 0 \\
0 & \Pi'_j\begin{pmatrix}
I \\
E_j
\end{pmatrix}
\end{pmatrix}.
\end{equation*}
\end{itemize}
Now, if $j$ is a leaf, this last display is precisely the HSS generator. If $j$ is not a leaf, we set $R_{c_1(j)}=U_j|_{\left({\bf i}'_j\cap{\bf i}_{c_1(j)}\right)\times\left[1:\left\vert{\bf\overline{i}}'_j\right\vert+r\right]}$ and $R_{c_2(j)}=U_j|_{\left({\bf i}'_j\cap{\bf i}_{c_2(j)}\right)\times\left[1:\left\vert{\bf\overline{i}}'_j\right\vert+r\right]}$.

\subsubsection{Constructing the remaining HSS generators}
\label{sec:resthss}
Now, note that for each $j$ at the leaf level in $\mathcal{T}$, each matrix $\left(T^-_j\right)_{({\bf i}'_{c_1(j)}\cup{\bf i}'_{c_2(j)})\times[1:n-|{\bf i}_j|]}$ used to obtain the generator $U_j$ yields the same $U_j$ regardless of the specific value of $j$. Hence, ${\bf i}_j'$ is the same for any leaf-level $j$. Therefore we can show by induction on $L$ that for each $j$ at the same depth of $\mathcal{T}$, $U_j$ and ${\bf i}'_j$ are the same. This shows that {\it we only need to perform the above steps once at each depth of $\mathcal{T}$} to obtain all the HSS row generators $U_j$ for a leaf-level $j$ and $R_j$ for $j$ with $\mathrm{depth}(j)\leq L-2$. Furthermore, because the above steps do not depend on the specific function $k(x,y)=f(|x-y|)$ as long as $f$ satisfies the analyticity condition, {\it the above steps also construct the HSS column generators} $V_j$ and $W_j$. So, we set $V_j=U_j$ for a leaf-level $j$ and $W_j=R_j$ for $j$ with $\mathrm{depth}(j)\leq L-2$. This last fact shows why our assumption that $f_1=f_2$ at the beginning of this section confers no loss of generality. Finally, for each $j\in\mathcal{T}$, we set $B_j=T_{{\bf i}'_j\times{\bf i}'_{\mathrm{sib}(j)}}$.

So far, we have not mentioned how many proxy points are required for the far-field approximation at each level in the above construction method; we will explore this issue in the next section. We note here, however, that if the number of proxy points is $O(\log n)$, then the flop count of this method is the same as that of the method in \cite{hypercauchy}, for a total of $O(\log^5n)$ flops. We will show that this is indeed the case in the next section whenever $f$ satisfies the univalence condition in Proposition~\ref{prop:toepuniv}.

\subsection{Number of proxy points required}
\label{sec:numberproxy}
First, we fix some notation: let $\mathcal{T},\mathcal{I}$ be the HSS tree and HSS index set of $T$, respectively, and let $j\in\mathcal{T}$ have corresponding index set ${\bf i}_j\in\mathcal{I}$. We define ${\bf\hat{i}}_j$ to be the subset of ${\bf i}_j$ missing its least and greatest $|{\bf i}_j|/4$ elements, ordered the usual way. We also define $\tilde{T}_n^{j,p}$ to be the $p$-point proxy point approximation (in the first variable) to the subblock $T\vert_{{\bf\hat{i}}_j\times[1:n]\setminus{\bf i}_j}=k({\bf\hat{i}}_j,[1:n]\setminus{\bf i}_j)$ with center $(1/2)\left(\min({\bf i}_j)+\max({\bf i}_j)\right)$ and radius $(1/2)\left(\max({\bf i}_j)-\min({\bf i}_j)+1\right)$.

Next, we show with Example~\ref{ex:cosproxypoint} that for general $f\in\mathcal{O}\left(\mathcal{B}\left(n/2,n/2\right)\right)$, this approximation need not have good convergence properties. This corresponds to the case that $f$ grows rapidly away from $n/2$; this corresponds to the case that the function bound in Equation~\eqref{eq:proxtol} is large.

\begin{example}
\label{ex:cosproxypoint}
For $n\geq8$, let $T_n\in\mathbb{R}_{n\times n}$ have entries $\left(T_n\right)_{i,j}=\cos\left((\pi/4)\left\vert j-i\right\vert\right)$, and let $\mathcal{I}_n=\{{\bf i}_{n,1},{\bf i}_{n,2},{\bf i}_{n,3}\}$ be its one-level HSS index set, indexed the usual way. Then the associated function $f(z)=f_1(z)=f_2(z)=\cos\left((\pi z)/4\right)$ is holomorphic on $\mathcal{B}\left(n/2,n/2\right)$. Table~\ref{tab:cosproxypoint} shows the minimum number of points $p$ required for $\tilde{T}_n^{1,p}$ to approximate $\left(T_n\right)\vert_{{\bf\hat{i}}_{n,1}\times[1:n]\setminus{\bf i}_{n,1}}$ to a given tolerance. Note that even for such small matrix sizes and large tolerance, the number of proxy points required already scales linearly with $n$. It is also worth noting that the rank of $T_n$ is at most 8 for all $n$ and every off-diagonal block.

\begin{table}[h]
\centering
\caption{The size $n$ of the matrix $T_n$ and the minimum number of proxy points $p$ required to attain $\left\Vert\left(T_n\right)\vert_{{\bf\hat{i}}_{n,1}\times[1:n]\setminus{\bf i}_{n,1}}-\tilde{T}_n^{1,p}\right\Vert_F<10^{-6}$.}
\vspace{.5cm}
\begin{tabular}{|p{.2in}||p{.2in}|p{.2in}|p{.2in}|p{.2in}|p{.2in}|p{.2in}|p{.2in}|p{.2in}|p{.2in}|}
\hline
$n$ & 16 & 24 & 32 & 40 & 48 & 56 & 64 & 72 & 80\\
\hline
$p$ & 21 & 27 & 34 & 39 & 47 & 53 & 59 & 65 & 72\\
\hline
\end{tabular}
\label{tab:cosproxypoint}
\end{table}
\end{example}
The poor performance in Example~\ref{ex:cosproxypoint} makes sense in light of Proposition~\ref{prop:proxelwisebound}: for each $y\in Y=[1:n]\setminus{\bf i}_{n,1}=\left[n/2+1:n\right]$, $k(z,y)=f\left(\vert z-y\vert\right)$ must not be too large in absolute value for all $z\in\partial F=\partial \mathcal{B}\left(n/4+1/2,\sqrt[4]{8}n/8\right)$ in order for a small number of proxy points to be sufficient. But in this case, we may observe that, if $y=n/2+1$, the maximum of $f\left(\vert y-z\vert\right)=\cos\left((\pi/4)\vert y - z\vert\right)$ along $z\in\partial F$ grows exponentially in $n$. In particular, even though cosine is bounded on the real line, its growth along the one-dimensional line $z(t)=t+it$ (for real $t$) is exponential. Hence, the growth of $p$ with respect to $n$ shown in Table~\ref{tab:cosproxypoint} gives evidence that $f$ with large values on $\mathcal{B}\left(n/2,n/2\right)$ may require a lot of proxy points for an accurate approximation.

On the other hand, if $f$ is bounded on the real line and univalent on $\mathcal{B}\left(n/2,n/2\right)$, we show in Example~\ref{ex:cosnproxypoint} that we do seem to have good proxy point convergence for the HSS approximation outlined in Sections~\ref{sec:leafhss} and \ref{sec:resthss}.
\begin{example}
\label{ex:cosnproxypoint}
For $n\geq8$, let $T_n\in\mathbb{R}_{n\times n}$ have entries $\left(T_n\right)_{i,j}=\cos\left((\pi\left\vert j-i\right\vert)/n\right)$. Then the associated function $f(z)=f_1(z)=f_2(z)=\cos\left((\pi z)/n\right)$ is univalent on $\mathcal{B}\left(n/2,n/2\right)$ and bounded on the real line. Table \ref{tab:cosnproxypoint} shows the minimum number of proxy points required for the sublinear HSS construction method to yield a given approximation tolerance for the topmost HSS row block.
\begin{table}[h]
\centering
\caption{The size $n$ of the matrix $T_n$ and the minimum value of $p$ such that the $L$-level HSS approximation constructed in Sections~\ref{sec:leafhss} and \ref{sec:resthss} with $p$ proxy points approximates the topmost HSS block of $T_n$ to a relative Frobenius norm error $10^{-10}$.}
\vspace{.7cm}
\begin{tabular}{|p{.12in}||p{.27in}|p{.27in}|p{.27in}|p{.27in}|p{.27in}|p{.27in}|p{.34in}|p{.34in}|p{.34in}|p{.34in}|}
\hline
$n$ & 2048 & 4096 & 4096 & 8192 & 8192 & 8192 & 16384 & 16384 & 16384 & 16384 \\
\hline
$L$ & 1 & 1 & 2 & 1 & 2 & 3 & 1 & 2 & 3 & 4 \\
\hline
$p$ & 26 & 27 & 27 & 28 & 28 & 28 & 28 & 28 & 28 & 28 \\
\hline
\end{tabular}
\label{tab:cosnproxypoint}
\end{table}
\end{example}
Example~\ref{ex:cosnproxypoint} gives numerical evidence that the proxy-point approximation has good enough convergence properties to be used in practice, even despite global HSS error accumulation. We now show that good proxy point convergence is true for general univalent $f$ in this context, as well as in the general case of Proposition~\ref{prop:fnbound}.

\begin{lemma}
\label{lem:integralbound}
Let $\mathcal{I}$ be an HSS index set for an $n\times n$ matrix, where $n$ is a power of 2; let ${\bf i}\in\mathcal{I}$; and let $l$ be the height of ${\bf i}$. Define $k(x,y)=f\left(\vert y-x\vert\right)$ for some $f\in\mathcal{O}\left(\mathcal{B}\left(n/2,n/2\right)\right)$; let $x\in{\bf\hat{i}}$; let $y\in[1:n]\setminus{\bf i}$; and let $p\in\mathbb{N}$. Then
\begin{equation}
\left\vert k(x,y)-\sum_{j=1}^p\left(\frac{\left(\sqrt[4]{8}\right)2^{l-1}}{p}\right)\frac{\omega^jk\left(z_j,y\right)}{z_j-x}\right\vert<14\frac{\max_{z\in\partial F}\left(\left\vert f\left(y-z\right)\right\vert\right)}{2^{p/4}-1},
\end{equation}
where $z_j=c+\left(\sqrt[4]{8}\right)2^{l-1}\omega^j$, $F$ is the open ball with center $c$ and radius $\left(\sqrt[4]{8}\right)2^{l-1}$, and $c=(1/2)\left(\max({\bf i})-\min({\bf i})+1\right)$.
\end{lemma}
\begin{proof}
This is a straightforward application of Proposition~\ref{prop:proxelwisebound}, where we set $X~=~{\bf\hat{i}}$; $Y=[1:n]\setminus{\bf i}$; and $D$ and $E$ to be the open balls with center $c$ and radii $R=2^{l-1}$ and $r=2^l$, respectively. We thus get $\alpha=2\sqrt[4]{2}/(\sqrt[4]{2}-1)<14$.
\end{proof}

Therefore, by the maximum modulus principle and Lemma~\ref{lem:integralbound}, if we find that $\max_{z\in\partial\mathcal{B}\left((n+1)/2,n/2-1\right)}\left\vert f(z)\right\vert$ has a sufficiently small bound with respect to $n$, we would need only $O\left(\log n\right)+\vert\log\epsilon\vert$ proxy points to obtain an entrywise proxy point approximation with tolerance $\epsilon$ at every height of the HSS tree. But note that we obtained exactly such a bound in Section~\ref{sec:proxy} in Proposition~\ref{prop:toepuniv} if $f$ is univalent on $\mathcal{B}\left(n/2,n/2\right)$, and if $f$ and its derivative does not grow too quickly quickly with respect to $n$ along the real axis. Hence, we obtain the following absolute error bound for the proxy point approximation of an off-diagonal ``far-field" row block:
\begin{corollary}
\label{cor:hssblockbound}
Let $T$ be the $n\times n$ matrix with entries $T_{i,j}=f\left(\vert j-i\vert\right)$, where $f\in\mathcal{O}\left(\mathcal{B}\left(n/2,n/2\right)\right)$ is injective on $\mathcal{B}\left(n/2,n/2\right)$. Let $\mathcal{I}$ be the HSS index set of $T$, and let ${\bf i}_j\in\mathcal{I}$. Then
\begin{equation*}
\left\Vert T\vert_{{\bf\hat{i}}_j\times[1:n]\setminus{\bf i}_j}-\tilde{T}^{j,p}\right\Vert_F\leq\left(\frac{7n^2}{2^{p/4+1}-2}\right)\left((n^3/8)\left\vert f'\left(n/2\right)\right\vert+\left\vert f\left(n/2\right)\right\vert\right).
\end{equation*}
\end{corollary}
\begin{proof}
By Lemma~\ref{lem:integralbound}, the maximum modulus principle, and Proposition~\ref{prop:fnbound}, in that order, we have that for each $1\leq u\leq\vert{\bf\hat{i}}_j\vert$ and $1\leq v\leq\left\vert[1:n]\setminus{\bf i}_j\right\vert$,
\begin{align*}
\left\vert \left(T\vert_{{\bf\hat{i}}_j\times[1:n]\setminus{\bf i}_j}\right)_{u,v}-\left(\tilde{T}^{j,p}\right)_{u,v}\right\vert&<14\frac{\max_{y\in[1:n]\setminus{\bf i},\;z\in\partial F}\left(\left\vert f\left(y-z\right)\right\vert\right)}{2^{p/4}-1}\\
&\leq14\frac{\max_{z\in\partial\mathcal{B}\left((n+1)/2,n/2-1\right)}\left(\left\vert f\left(z\right)\right\vert\right)}{2^{p/4}-1}\\
&\leq\frac{14}{2^{p/4}-1}\left((n^3/8)\left\vert f'\left(n/2\right)\right\vert+\left\vert f\left(n/2\right)\right\vert\right).
\end{align*}
Since $\vert{\bf\hat{i}}_j\vert,\left\vert[1,n]\setminus{\bf i}_j\right\vert\leq\frac{n}{2}$, the result follows by summing over all $u$ and $v$.
\end{proof}
Thus, to obtain a given proxy point approximation tolerance $\epsilon$ for any level, we need $O\left(\log n\right)+O\left(\left\vert f\left(n/2\right)\right\vert\right)+O\left(\left\vert f'\left(n/2\right)\right\vert\right)+O\left(\vert\log\epsilon\vert\right)$ proxy points. In practice, $f$ and its derivative are often bounded on the real line, as in Examples \ref{ex:nonunivalent} and \ref{ex:cauchy} below.

\section{Discussion and numerical tests}
\label{sec:numerical}
First, we note that, although injectivity of $f$ as defined in the previous section is a sufficient condition, it is not strictly necessary in practice to enable the use of our sublinear Toeplitz HSS construction algorithm. The point of the injectivity criterion is simply to allow, using Proposition~\ref{prop:fnbound}, a sufficiently slow growth bound for $f$ that depends only on its radius of analyticity. However, functions $f$ that are not univalent on the relevant region can also grow sufficiently slowly in order for their related construction algorithm outlined in the previous section to work on the related Toeplitz matrix. Example~\ref{ex:nonunivalent} illustrates this.

\begin{example} For $n\geq8$, let $T_n\in\mathbb{R}_{n\times n}$ have entries $\left(T_n\right)_{i,j}=\left(\left\vert j-i\right\vert-n/2\right)^2$, so the associated function $f(z)=f_1(z)=f_2(z)=\left(z-n/2\right)^2$ is not univalent on $\mathcal{B}\left(n/2,n/2\right)$. Table~\ref{tab:nonunivalent} lists the relative approximation tolerance for various HSS approximations of $T$ from Sections~\ref{sec:leafhss} and \ref{sec:resthss}. (For the scheme as outlined there, we set the maximum off-diagonal rank to $r=28$. This is sufficient, since each matrix involved has a relative off-diagonal numerical rank of 3 with respect to the tolerance $10^{-14}$.) Note that relatively small values of $p$ result in a good approximation.
\label{ex:nonunivalent}
\end{example}
\begin{table}[h]
\centering
\caption{The relative Frobenius norm errors of the $L$-level HSS approximation to $T_n$ from Sections~\ref{sec:leafhss} and \ref{sec:resthss} using $p$ proxy points. The top and bottom tables show the errors using 32 and 48 proxy points at each level, respectively.}
\hspace{-.23cm}
\begin{tabular}{|p{1.11in}||p{.38in}|p{.38in}|p{.38in}|p{.38in}|p{.38in}|p{.38in}|}
\hline
$n$ & 2048 & 2048 & 8192 & 8192 & 16384 & 16384 \\
\hline
$L$ & 2 & 4  & 4 & 6 & 6 & 7 \\
\hline
rel. err. (e$10^{-13}$) & $5.4863$ & $2.9697$ & $7.7119$ & $3.3541$ & $6.9370$ & $3.4362$\\
\hline
\end{tabular}\vspace{.02in}

\begin{tabular}{|p{1.11in}||p{.38in}|p{.38in}|p{.38in}|p{.38in}|p{.38in}|p{.38in}|}
\hline
$n$ & 2048 & 2048 & 8192 & 8192 & 16384 & 16384 \\
\hline
$L$ & 2 & 4 & 4 & 6 & 6 & 7 \\
\hline
rel. err. (e$10^{-13}$) & $2.0441$ & $9.3656$ & $3.2532$ & $1.0675$ & $2.9239$ & $1.0933$ \\
\hline
\end{tabular}
\label{tab:nonunivalent}
\end{table}

On the other hand, the conditions of Proposition~\ref{cor:hssblockbound} provides a wide class of functions for which our sublinear HSS construction algorithm is guaranteed to work.

\begin{example} Since $f_1(z)=n/z$ and $f_2(z)=-n/z$ are univalent on $\mathcal{B}\left(n/2,n/2\right)$, the method from Sections~\ref{sec:leafhss} and \ref{sec:resthss} should work to find the HSS generators of $T_n$, the Cauchy kernel matrix evaluated at $n$ equidistant points in $[-1,1]$, in sublinear time. Table~\ref{tab:cauchy} lists the relative approximation tolerance for various HSS approximations to the matrix $T_n\in\mathbb{R}_{n\times n}$ with off-diagonal values $\left(T_n\right)_{i,j}=n/(j-i)$ and diagonal values equal to 0. The maximum relative off-diagonal numerical rank $r$ is also listed; for this experiment, we set $r=28$ for each matrix. It is worth noting that the accuracy bound given in \cite{kercompr} may also be used in lieu of Proposition~\ref{prop:proxelwisebound} for this particular kernel matrix to indicate applicability of the scheme from Section~\ref{sec:toeplitzconst}.
\label{ex:cauchy}
\end{example}
\begin{table}[h]
\centering
\caption{The relative Frobenius norm errors of the $L$-level HSS approximation to $T_n$ from Sections~\ref{sec:leafhss} and \ref{sec:resthss} using $p$ proxy points, as well as the numerical HSS rank $r$ of $T_n$ with tolerance $10^{-14}$. Again, the top and bottom tables show the errors using 32 and 48 proxy points at each level, respectively.}
\hspace{-.23cm}
\begin{tabular}{|p{1.11in}||p{.38in}|p{.38in}|p{.38in}|p{.38in}|p{.38in}|p{.38in}|}
\hline
$n$ & 2048 & 2048 & 8192 & 8192 & 16384 & 16384 \\
\hline
$r$ & 26 & 26 & 30 & 30 & 33 & 33 \\
\hline
$L$ & 2 & 4 & 4 & 6 & 6 & 7 \\
\hline
rel. err. (e$10^{-14}$) & $7.1041$ & $5.9208$ & $8.1024$ & $6.1210$ & $9.4705$ & $6.1585$ \\
\hline
\end{tabular}\vspace{.02in}

\begin{tabular}{|p{1.11in}||p{.38in}|p{.38in}|p{.38in}|p{.38in}|p{.38in}|p{.38in}|}
\hline
$n$ & 2048 & 2048 & 8192 & 8192 & 16384 & 16384 \\
\hline
$r$ & 26 & 26 & 30 & 30 & 33 & 33 \\
\hline
$L$ & 2 & 4 & 4 & 6 & 6 & 7 \\
\hline
rel. err. (e$10^{-14}$) & $1.7926$ & $1.1841$ & $2.1102$ & $1.2407$ & $2.5062$ & $1.2521$ \\
\hline
\end{tabular}
\label{tab:cauchy}
\end{table}

Again, we note that even after global error accumulation associated with an HSS tree of depth 6 and 7 in Examples~\ref{ex:cauchy} and \ref{ex:nonunivalent}, the relative error is still quite low. This gives evidence that the asymptotic error decay regime from Proposition~\ref{prop:fnbound} holds well enough in practice: note that the maximum of the function in Example~\ref{ex:nonunivalent} is even increasing on $\mathcal{B}(n/2,n/2)$ as $n$ grows. This increase, however, is polynomial in $n$, and therefore so is the numerator of the bound given by Corollary~\ref{cor:hssblockbound}. The denominator of this bound is exponential in $p$, which helps explain the quality of the approximation in Example~\ref{ex:nonunivalent}.

\section{Extensions}
\label{sec:ext}
In a forthcoming study \cite{horninglepilov}, we use the arguments of Section~\ref{sec:numberproxy} to bound the numerical rank of certain classes of matrices. Specifically, we may conformally map the complement of the sets of points $X$ and $Y$ to improve their separation ratio $R/r$ in the quadrature rule error given by Proposition~\ref{prop:proxelwisebound}. This allows us to obtain bounds similar to, but much better than, those given by applying Corollary~\ref{cor:hssblockbound}, and in turn it enables us to argue when certain one-dimensional kernel matrices may have low numerical rank.

Furthermore, it is also possible to perform a more detailed analysis of the global error accumulated after all compression steps in Sections~\ref{sec:leafhss} and \ref{sec:resthss} are performed, including the SRRQR factorization steps. This would give additional motivation for proving an absolute bound in Proposition~\ref{prop:proxelwisebound}, Proposition~\ref{prop:toepuniv}, and Corollary~\ref{cor:hssblockbound}, since relative bounds are harder to integrate into a global HSS error analysis.

Finally, we may also extend the bound of Proposition~\ref{prop:proxelwisebound} to analytic functions of more than one (complex) variable. In particular, no part of the argument used in this proposition relies on complex analysis concepts that apply only in the one-variable case. Hence, we may explore generalizations of the complex-analytic low-rank approximations discussed here to more general Toeplitz matrices, as well as to non-Toeplitz matrices that are defined by analytic functions in other ways. When doing so, we may also combine the results of Section~\ref{sec:numberproxy} with the hierarchical partitioning described in \cite{mhs}.

\end{document}